\documentclass[preprint,sort&compress,12pt]{elsarticle}
\pdfoutput=1 
\usepackage[hidelinks, colorlinks=true]{hyperref}
\usepackage{amscd,amssymb,amsmath,amsthm}
\usepackage[all]{xy}
\usepackage{array}
\usepackage{etoolbox}
\usepackage{mathtools}
\usepackage{siunitx}
\usepackage{multirow}
\usepackage{tikz-cd}
\usepackage{tikz}
\usepackage{tabularx}
\usetikzlibrary{matrix,fit}
\usetikzlibrary{positioning}
\usetikzlibrary{shapes.geometric}
\usepackage{caption}
\usepackage[tickmarkheight=0.1cm]{todonotes}

\newtheorem*{th33}{Theorem 3.3}
\newtheorem*{th34}{Corollary 3.4}
\newtheorem*{th35}{Theorem 3.5}
\newtheorem*{th38}{Theorem 3.8}
\newtheorem*{th43}{Theorem 4.3}
\newtheorem{theorem}{Theorem}[section]

\newtheorem{lemma}[theorem]{Lemma}
\newtheorem{corollary}[theorem]{Corollary}
\newtheorem{prop}[theorem]{Proposition}

\newcommand{\mybinom}[3][0.8]{\scalebox{#1}{$\dbinom{#2}{#3}$}}
\newtheorem{Remark}{Remark}
\numberwithin{equation}{section}
\theoremstyle{definition}
\newtheorem{definition}[theorem]{Definition}

\newtheorem*{Outline}{Outline of the paper}
\newtheorem*{Notation}{Notation}

\newlength\mylen
\usepackage{array} 
\newcolumntype{C}{>{\hfil$}p{\mylen}<{$\hfil}}

\newcommand\dela[1]{}

\newcommand{\im}[1]{\text{Im}{(#1)}}

\usepackage[utf8]{inputenc}

\makeatletter
\def\ps@pprintTitle{%
	\let\@oddhead\@empty
	\let\@evenhead\@empty
	\def\@oddfoot{}%
	\let\@evenfoot\@oddfoot}
\makeatother

\begin{document}
	
	\begin{frontmatter}
		
		\title{On the size of the Schur multiplier of finite groups}
		 \author[IISER TVM]{Sathasivam Kalithasan}
            \ead{sathasivam19@iisertvm.ac.in}
		\author[IISER TVM]{Tony Nixon Mavely}
		\ead{tonynixonmavely17@iisertvm.ac.in}
		\author[IISER TVM]{Viji Zachariah Thomas\corref{cor1}}
		\address[IISER TVM]{School of Mathematics,  Indian Institute of Science Education and Research Thiruvananthapuram,\\695551
			Kerala, India.}
		\ead{vthomas@iisertvm.ac.in}
		\cortext[cor1]{Corresponding author. \emph{Phone number}: +91 8921458330}
		
		\begin{abstract}
        We obtain bounds for the size of the Schur multiplier of finite $p$-groups and finite groups, which improve all existing bounds. Moreover, we obtain bounds for the size of the second cohomology group $H^2(G,\mathbb{Z}/p\mathbb{Z})$  of a $p$-group with coefficients in $\mathbb{Z}/p\mathbb{Z}$. Denoting the minimal number of generators of a $p$-group $G$ by $d(G)$, our bound depends on the parameters $|G|=p^n$, $|\gamma_2G|=p^k$, $d(G)=d$, $d(G/Z)=\delta$ and $d(\gamma_2G/\gamma_3G)=k'$. For special $p$-groups, we further improve our bound when $\delta-1 > k'$. Moreover, given natural numbers $d$, $\delta$, $k$ and $k'$ satisfying $k=k'$ and $\delta-1 \leq k'$, we construct a capable $p$-group $H$ of nilpotency class two and exponent $p$ such that the size of the Schur multiplier attains our bound.
		\end{abstract}

		\begin{keyword}
		 Schur multiplier \sep $p$-groups \sep Special $p$-groups \sep Nilpotent groups  \sep  Alternating bilinear map\sep Graphs \sep Second cohomology group \sep Capable group.
   
			\MSC[2020]   20J05 \sep 20D15 \sep 20F18 \sep 20J06 
		\end{keyword}
		
	\end{frontmatter}

\section{Introduction} 
The study of the Schur multiplier of a $p$-group, and bounding its size, have attracted many researchers over the years \cite{BeyTap1982,Elli1998,EllWie1999,GasNeuYen1967,Gree1956,Hatu2017,Hatu2020,Jone1973,Karp1987,MavTho2023,Mora2011,Rai2017,Rai2018,Rai2024, ShaNirJoh2020, Verm1974}.
The aim of the present paper is to give a bound for the size of
 Schur multiplier of finite groups which improves existing bounds.  One of the main tools used to bound the size of the Schur multiplier of a $p$-group in recent years is the following beautiful inequality \cite[Proposition 1]{EllWie1999} of Ellis and Wiegold, 
 \begin{align}\label{EllisInequality}
    |M(G)||\gamma_2G||\im{\Psi_2}| \leq |M(G^{ab})|\prod^c_{i=2}|\gamma_iG/\gamma_{i+1}G \otimes \overline{G}^{ab}|,
\end{align}
where $\Psi_2$ is defined in \eqref{eqn:Psi_2_definition}.
  In recent years, the study of the size of the Schur multiplier for special $p$-groups of different ranks \cite{BeyTap1982, Rai2018, Hatu2020, MavTho2023} has received considerable interest. For example, Berkovich and Janko asked to find the Schur multiplier of special groups of rank $2$ and rank $\binom{d}{2}$ (Problems $2027$ and $1729$ of \cite{BerJan2011}, respectively). These questions were addressed in \cite{Hatu2020} and \cite{Rai2018}. It is natural to ask about the size of the Schur multiplier of intermediate ranks. The authors of \cite[Theorem 3.2]{MavTho2023}  achieved this for all ranks $(2 \leq k \leq \binom{d}{2})$ simultaneously. In \cite[Theorem 1.1]{Rai2018}, Rai proved the following inequality
\begin{align}
        |M(G)| \leq p^{\frac{1}{3}(d-1)d(d+1)}\label{eq:Raispecialbound}
\end{align}
for special $p$-groups with $d(G)=d \geq 3$ and $|\gamma_2G|=p^{\frac{d(d-1)}{2}}$. More recently, Rai proved the following theorem using a modification of \eqref{EllisInequality}.
  
    \begin{theorem}(\cite[Theorem 1.4]{Rai2024})\label{RaiimprovesEllisWiegold}
        Let $G$ be a non-abelian $p$-group of nilpotency class $c$  with $|\gamma_2G|=p^k$, $d(\gamma_2G/\gamma_3G)=k'$ and $d=d(G)$. The following inequality holds:
        \begin{align}
             |M(G)| \leq p^{\frac{1}{2}(d-1)(n+k)-\sum \limits_{i=2}^{\min (d,k'+1)} d-i}. \label{eq:Raitheorembound}
        \end{align}     
    \end{theorem}
    Moreover, in the proof of the above theorem, the author obtains the bound
    \begin{align}
        |M(G)| \leq p^{\frac{1}{2}(d-1)(n+k)-k(d-\delta)-\sum \limits_{i=2}^{\min (\delta,k'+1)} \delta-i},\label{schurinequality}
    \end{align}
    where $\delta=d(G/Z)$. The bound in \eqref{schurinequality} is clearly an improvement over the bound in \eqref{eq:Raitheorembound}.
    
    It is easy to see that when Theorem \ref{RaiimprovesEllisWiegold} is restricted to the class of special $p$-groups with $|\gamma_2G|=p^{\frac{1}{2}d(d-1)}$, one does not obtain the bound given in \eqref{eq:Raispecialbound}. In fact, the bound in \eqref{eq:Raispecialbound} is stronger than the one given in Theorem \ref{RaiimprovesEllisWiegold}. Therefore, Theorem \ref{RaiimprovesEllisWiegold} cannot be considered as a generalization of \eqref{eq:Raispecialbound}.

     The main objective of this paper is to establish a bound on the size of the Schur multiplier of $p$-groups that improves all existing bounds, and when restricted to the class of special $p$-groups, provides the strongest known bound for that class. Our next theorem achieves this and improves the bound given in Theorem \ref{RaiimprovesEllisWiegold} (cf. Lemma \ref{BoundComparison}), and it also matches the bound given in \eqref{eq:Raispecialbound} when restricted to the class of special $p$-groups with $|\gamma_2G|=p^{\frac{1}{2}d(d-1)}$. Moreover, it is worth noting that when Theorem \ref{mainbound} is restricted to the class of special $p$-groups, it gives the same bound obtained in \cite[Theorem 3.2]{MavTho2023}. Hence, it is a generalization of the bounds given in \cite{Rai2024, MavTho2023, Hatu2020, Rai2018, Rai2017}.

    \begin{th33}
            Let $G$ be a $p$-group of order $p^n$. Assume that $d(G)=d$, $d(G/Z)=\delta$, $d(\gamma_2G/\gamma_3G)=k'$ and $|\gamma_2G|=p^k$. Let $r$ and $t$ be non-negative integers such that ${\delta \choose 2}-k'={r \choose 2}+t$, where $0 \leq t < r$. Then,
            \[|M(G)| \leq p^{\frac{1}{2}(d-1)(n+k)-k(d-\delta)-\binom{\delta}{3}+ \binom{r}{3} + \binom{t}{2}}.\]
        \end{th33}
Note that $r$ and $t$ are determined by $\delta$ and $k'$. In addition, the function $f(\delta,k')=\binom{r}{3}+\binom{t}{2}$ seems to connect to the function in \cite[Theorem 5.24]{Magi2014}, which is itself related to a numerical sufficient condition for capability of a $p$-group of class two and exponent $p$, a property that is known to have connections to the Schur multiplier. In the next corollary, we give a bound for the size of the second cohomology group $H^2(G,\mathbb{Z}/p\mathbb{Z})$  of a $p$-group with coefficients in $\mathbb{Z}/p\mathbb{Z}$.
        \begin{th34}
            Let $G$ be a $p$-group of order $p^n$. Assume that $d(G)=d$, $d(G/Z)=\delta$, $d(\gamma_2G/\gamma_3G)=k'$ and $|\gamma_2G|=p^k$. Let $r$ and $t$ be non-negative integers such that ${\delta \choose 2}-k'={r \choose 2}+t$, where $0 \leq t < r$. Then,
            \[|H^2(G,\mathbb{Z}/p\mathbb{Z})| \leq p^{\frac{1}{2}(d-1)(n+k)-k(d-\delta)-\binom{\delta}{3}+ \binom{r}{3} + \binom{t}{2}+ d }.\]
        \end{th34}
The bounds given in Theorem \ref{mainbound} and Corollary \ref{main bound Zp coefficient } can be achieved for certain values of $d$, $\delta$, $k$ and $k'$. Consider the following $p$-group of nilpotency class $2$ with $d=\delta=6$ and $k=k'=4$,
  \begin{gather*}
       G=\langle g_1,g_2,g_3,g_4,g_5,g_6,q_1,q_2,q_3,q_4 \mid  g_i^3=q_j^3=[q_j,q_k]=1 \text{ for all } 1\leq i \leq 6 \\
        \qquad\text{and } 1 \leq j,k \leq 4, [g_1,g_2]=[g_1,g_6]=q_1, \  [g_1,g_3]=q_2, \ [g_3,g_4]=q_3, \\
      \quad [g_1,g_5]=q_4,  \ [g_i,g_j]=1 \text{ for all other } 1 \leq i < j \leq 6\rangle.
    \end{gather*}
    Using GAP \cite{GAP4}, it can be verified that $|M(G)|=3^{23}$ and $|H^2(G,\mathbb{Z}/3\mathbb{Z})|=3^{29}$, which is the same as the bounds given by Theorem \ref{mainbound} and Corollary \ref{main bound Zp coefficient }, respectively.

    In the next theorem, we extend Theorem \ref{mainbound} to all finite groups.
\begin{th35}
    Let G be a finite group of order $n$ and let $G_p$ be a Sylow $p$-subgroup of $G$ of order $p^{n_p}$. Assume that $d(G_p)=d_p$, $d(G_p/Z(G_p))=\delta_p$, $d(\gamma_2G_p/\gamma_3G_p)=k'_p$ and $|\gamma_2G_p|=p^{k_p}$. Let $r_p$ and $t_p$ be non-negative integers such that ${\delta_p \choose 2}-k'_p={r_p \choose 2}+t_p$, where $0 \leq t_p < r_p$. Then,
            \[|M(G)| \leq \prod p^{\frac{1}{2}(d_p-1)(n_p+k_p)-k_p(d_p-\delta_p)-\binom{\delta_p}{3}+ \binom{r_p}{3} + \binom{t_p}{2}}
            ,\] where the product runs over all the primes dividing $n$.
\end{th35}

It is worth noting that the bound given in Theorem \ref{general main bound} is also attained for some values of $d_p$, $\delta_p$, $k_p$ and $k'_p$.
Consider,
       \begin{gather*}
 H=\langle h_1,h_2,h_3,h_4,h_5,h_6,r_1,r_2,r_3,r_4 \mid h_i^5=r_j^5=[r_j,r_k]=1 \text{ for all } 1\leq i \leq 6 \\
          \qquad \text{ and } 1 \leq j,k \leq 4, [h_1,h_2]=[h_1,h_6]=r_1, \ [h_1,h_3]=r_2, \ [h_3,h_4]=r_3,  \\
         [h_1,h_5]=r_4, \ [h_i,h_j]=1 \text{ for all other } 1 \leq i < j \leq 6 \rangle.
    \end{gather*}
 Set $K=G\times H$, where $G$ is the group defined previously. It can be verified using GAP \cite{GAP4}, that $|M(K)|=3^{23}5^{23}$, which is the same as the bound given by Theorem \ref{general main bound}.

         If $\delta < k'+1$, then the bound in Theorem \ref{mainbound} improves the one given in Theorem \ref{RaiimprovesEllisWiegold}. If $\delta \geq k'+1$, then Theorem \ref{RaiimprovesEllisWiegold} and Theorem \ref{mainbound} give the same bound when restricted to special $p$-groups. The aim of the next theorem is to improve the bounds given in Theorems \ref{RaiimprovesEllisWiegold} and \ref{mainbound} for $\delta > k'+1$. Note that $d=\delta$, and $k=k'$ for special $p$-groups. 

          \begin{th38}
                Let $G$ be a special $p$-group of order $p^n$ with $d(G)=d$ and $d(\gamma_2G)=k> 2$. If $d>k+1$, then
                \[|M(G)| \leq p^{\frac{1}{2}(d-1)(n+k)-\big(\sum\limits_{i=2}^{k+1}(d-i)\big) -(k-2)}.\]
       \end{th38}
 Consider the following special $p$-group with $d=5$ and $k=3$,
\begin{gather*}
        G=\langle g_1,g_2,g_3,g_4,g_5,q_1,q_2,q_3 \mid g_i^3=q_j^3=[q_j,q_k]=1 \text{ for all } 1\leq i \leq 5 \text{ and } \\
        1 \leq j,k \leq 3,  [g_1,g_2]=[g_1,g_5]=[g_2,g_5]=q_1, \ [g_2,g_3]=q_2,    \\
          [g_3,g_4]=q_3, \ [g_i,g_j]=1 \text{ for all other } 1 \leq i < j \leq 5 \rangle.
    \end{gather*}
Using GAP \cite{GAP4}, it can be verified that $|M(G)|=3^{14}$. The bound given by Theorem \ref{mainboundd>k'+1} is $3^{15}$, which indicates that our bound is close to the actual size of the Schur multiplier for these values of $d$ and $k$.
\begin{table}
    \centering
    	\begin{tabular}{ |c|c|c|c|c|c|c| } 
		\hline
            \multirow{2}*{Order} & \multirow{2}*{$d$} & \multirow{2}*{$\delta$} & \multirow{2}*{$k$} & \multirow{2}*{$k'$} &  \multicolumn{2}{|c|}{Bounds on $|M(G)|$}          \\
            & & & & & Theorem \ref{mainbound} & \cite[Theorem 1.4]{Rai2024} \\ 
		\hline
		$5^{21}$ & 7 & 6 & 14 & 14 & $5^{71}$ &  $5^{81}$  \\ 
		$5^{24}$ & 7 & 7 & 17 & 17 & $5^{89}$ &  $5^{108}$  \\ 
  		$3^{39}$ & 8 & 8 & 29 & 28 & $3^{182}$ & $3^{217}$\\  
		$3^{50}$ & 10 & 10 & 40 & 40 & $3^{287}$ & $3^{369}$\\  
		\hline
	\end{tabular}
    \captionof{table}[foo]{Comparison of bounds on the size of Schur multipliers of $p$-groups}
    \label{table}
\end{table}
There exist $p$-groups of order $5^{21}$, $5^{24}$ and $3^{50}$ of nilpotency class $2$ with the parameters $d$, $\delta$, $k$ and $k'$ given in the table such that the size of the Schur multiplier of the groups is equal to the corresponding bounds given by Theorem \ref{mainbound}. Table 1 shows that as $k'-\delta$ increases, our bound improves the recent bound in \cite{Rai2024}. 

We say that a group $G$ is capable if there exists a group $H$ such that $G \cong H/Z(H)$. 
The next theorem shows the bound in Theorem \ref{mainbound} is sharp for $p$-groups of exponent $p$ and class 2 satisfying $\delta - 1 \leq k'$ and $k=k'$.

    \begin{th43} 
       Given any natural numbers $d$, $\delta$, $k$, $k'$ satisfying $k=k'$ and $\delta -1 \leq k'$, there exists a capable $p$-group $H$ of nilpotency class two and exponent $p$ with $d(H)=d, $ $d(H/Z)=\delta$, $d(\gamma_2H)=k'$ and $\vert\gamma_2H\vert=p^k$ such that $\vert M(H)\vert$ attains the bound in {Theorem~\ref{mainbound}}.
    \end{th43}
    \begin{Remark}
        Theorem \ref{sharpness thm} shows that our bound in Theorem \ref{mainbound} is sharp for the class of $p$-groups of nilpotency class two with $\Phi(G)=\gamma_2G$ satisying ${\delta-1\leq k'}$. Note that for this class of $p$-groups, the size of the group is determined by the parameters $d$ and $k$. In particular, $n=d+k$.
    \end{Remark} 
\begin{Outline} In Section 2, we consider two vector spaces $U, \ V$ over a field $F$ and an alternating bilinear map $A:U\times U\to V$ such that $\im{A}$ spans $V$. We construct a special basis for $V$ consisting of elements of $\im{A}$ and then study several properties of this basis. Using the map $A$, we construct a trilinear alternating map $\Psi: U\times U\times U\to V\otimes U$, which induces a linear map $\Psi:U\otimes U\otimes U\to V\otimes U$. Using graph theory and some counting arguments, we estimate a lower bound for the dimension of $\im{\Psi}$. In Section 3, using this lower bound for the dimension of $ \im{\Psi}$ in \eqref{EllisInequality}, we obtain a bound on the size of the Schur multiplier given in Theorem \ref{mainbound}. In Corollary \ref{main bound Zp coefficient }, we obtain a bound for $|H^2(G,\mathbb{Z}/p\mathbb{Z})|$. In Theorem \ref{general main bound}, we obtain a bound for the size of the Schur multiplier for any finite group. In the rest of Section 3, we focus on special $p$-groups and improve the bounds given in Theorem \ref{mainbound} for $d>k+1$. In Section 4, we show that the bound obtained in Theorem \ref{mainbound} is sharp for certain choices of parameters $d$, $\delta$, $k$, and $k'$. In particular, for each choice of $d$, $\delta$, $k$ and $k'$ with $\delta-1 \leq k'$ and $k=k'$, we construct a capable $p$-group of nilpotency class two and exponent $p$ which attains the bound in Theorem \ref{mainbound}. 
\end{Outline}
	
\begin{Notation} We will denote the commutator subgroup of $G$ as $\gamma_2G$; $\gamma_iG$ will denote the $i^{th}$ term of the lower central series of $G$; the nilpotency class of $G$ will be denoted by $c$; $d(G)$ will denote the cardinality of a minimal generating set for $G$; $G^{ab}$ will denote the abelianization of $G$; $\Phi(G)$ will denote the Frattini subgroup of $G$; $Z(G)=Z$ will denote the center of the group; $\mathbb{F}_p$ will denote the finite field with $p$ elements. The Schur multiplier of $G$ will be denoted by $M(G)$.
\end{Notation}
\section{Constructing a special basis using elements in the image of an alternating map}
 Let $U$ and $V$ be finite-dimensional $F$ vector spaces and $A: U \times U \to V$ be an alternating bilinear map such that $\im {A}$ spans $V$.

 \subsection{Construction of a set associated with an alternating map } \label{Construction_Basis}  \vspace{.5em}

Let $X=\{u_1,\cdots,u_n\}$ be a basis for $U$ and $u_1<u_2<\dots<u_n$ be an ordering on $X$. Clearly $Y=\{A(u_i,u_j)\mid 1\leq i <j \leq n\}$ 
spans $V$.  We define a set $\mathcal{Y}=\{\{i,j\} \mid 1 \leq i, j \leq n, i \neq j\}$.
The total order on $X$  induces a total order on the set $\mathcal{Y}$ defined as $\{i,j\}<\{r,s\}$, whenever 
\begin{enumerate}
    \item $\max\{u_i,u_j\}<\max\{u_r,u_s\}$ or 
    \item $\max\{u_i,u_j\}=\max\{u_r,u_s\}=u_a$ and $\{u_i,u_j\}\setminus\{u_a\}<\{u_r,u_s
\}\setminus \{u_a\}$.
\end{enumerate}
The order relation on the set $\mathcal{Y}$ can be visualized as follows:
\[\begin{array}{cccccc}
     &\{1,2\} \\
     &\{1,3\} &\{2,3\}  \\
     &\{1,4\} &\{2,4\} &\{3,4\}\\
     &\cdots\\
     &\{1,n\}& \{2,n\} &\{3,n\}&\cdots &\{n-1,n\}.\\
\end{array}\]
\begin{sloppypar}
\noindent
Note that $\{i,j\}<\{r,s\}$ if $\{i,j\}$ is in one of the rows above $\{r,s\}$ or $\{i,j\}$ appears to the left of $\{r,s\}$ in the same row. Now we construct a set $\mathcal{B}\subseteq \mathcal {Y}$ associated with the alternating map. We begin the construction of $\mathcal{B}$ by choosing the least element $\{i_1,j_1\}$ of $\mathcal{Y}$ such that $A(u_{i_1},u_{j_1})$ is non-zero. Either $A(u_{i_1},u_{j_1}) \in Y$ or $A(u_{j_1},u_{i_1}) \in Y$, and without loss of generality we assume $A(u_{i_1},u_{j_1}) \in Y$. Consider the sets $\mathcal{B}_1=\{\{{i_1},{j_1}\}\}$ and $B_1=\{A(u_{i_1},u_{j_1})\}$. We define $\mathcal{B}_{k+1}$ and $B_{k+1}$ inductively as follows: Let $\{i_{k+1},j_{k+1}\}$ be the least element of $\mathcal{Y}$ such that $A(u_{i_{k+1}},u_{j_{k+1}})$ is not in the span of $B_k$. Again, without loss of generality, we assume that $A(u_{i_{k+1}},u_{j_{k+1}}) \in Y$. Define ${\mathcal{B}_{k+1}:=  \mathcal{B}_{k}\cup\{\{{i_{k+1}},j_{k+1}\}\}}$ and ${B_{k+1}:= B_k\cup \{A(u_{i_{k+1}},u_{j_{k+1}})\}}$. This process will stop at $m=\dim _{F}V$. Let $\mathcal{B}$ be the set obtained in the last step, that is, $\mathcal{B}=\mathcal{B}_m$. Note that the construction of $\mathcal{B}$ depends on $X$ and the order on $X$. Set $B=B_m$. In Lemma \ref{Bproperties}, we will show that $B$ is a basis of $V$. 
\end{sloppypar}
 Throughout this section, when a basis of $V$ is mentioned, it is assumed to be the basis constructed above. We give an example to show how $\mathcal{B}$ is constructed.  
 
\textbf{Example 1 (Construction of $\mathcal{B}$):} Consider the $\mathbb{F}_p$ vector spaces $U$ and $V$. Let $X=\{u_1,u_2,u_3,u_4,u_5\}$ be a basis of $U$ and $\{v_1,v_2,v_3,v_4\}$ be a basis of $V$. Assume that $A: U \times U \to V$ is an alternating map defined as follows,
    \begin{align*}
        A(u_1,u_2)&=v_1 \\
        A(u_1,u_3)&=v_2 & A(u_2,u_3)&=0  \\
        A(u_1,u_4)&=0 & A(u_2,u_4)&=v_3 & A(u_3,u_4)=0\\
        A(u_1,u_5)&=-(v_1+v_2) & A(u_2,u_5)&=v_4 & A(u_3,u_5)=0 & & A(u_4,u_5)=v_3.
    \end{align*}
    Let $X$ be ordered as $u_1<u_2<\cdots<u_5$. Observe that $\{1,2\}$ is the least element of $\mathcal{Y}$ such that $A(u_1,u_2)$ is non-zero. Thus, $\mathcal{B}_1=\{\{1,2\}\}$ and ${B_1=\{A(u_1,u_2)\}}$. Next, observe that $A(u_1,u_3) \notin \text{span }(B_1)$ and $\{1,3\}$ is the smallest such element in $\mathcal{Y}$. Thus, we have that $\mathcal{B}_2=\{\{1,2\},\{1,3\}\}$ and ${B_2=\{A(u_1,u_2),A(u_1,u_3)\}}$. Now, $A(u_2,u_4) \notin \text{span }(B_2)$ and $\{2,4\}$ is the smallest such element. Note that $\{4,5\}$ is another element such that ${A(u_4,u_5) \notin \text{span }(B_2)}$, but $\{4,5\}$ is not the smallest such element. Thus, ${\mathcal{B}_3=\{\{1,2\},\{1,3\},\{2,4\}\}}$ and $B_3=\{A(u_1,u_2),A(u_1,u_3),A(u_2,u_5)\}$. Lastly, $A(u_2,u_5) \notin \text{span }(B_3)$ and $\{2,5\}$ is the smallest such element. Thus, we have that ${\mathcal{B}_4=\{\{1,2\},\{1,3\},\{2,4\},\{2,5\}\}}$. Corresponding to $\mathcal{B}_4$, we have that $B_4=\{A(u_1,u_2),A(u_1,u_3),A(u_2,u_4),A(u_2,u_5)\}$. Since the dimension of $V$ over $\mathbb{F}_p$ is $4$, $B=B_4$ and $\mathcal{B}=\mathcal{B}_4$.

\subsection{The graph associated with the alternating map $A$}  \vspace{.5em}

In the previous section, given an order on the basis of $U$ and an alternating map  $A: U \times U \to V$, we constructed a set $\mathcal{B}$ associated with an alternating map. Using $\mathcal{B}$, we now define a graph $\mathcal{G}(\mathcal{B})$ associated with this alternating map and ordered basis.

\begin{definition} \label{graphofalternatingmap}
    The graph $\mathcal{G}(\mathcal{B})$ associated with the alternating map $A$ has $n$ vertices numbered $1, \ldots, n$ and an edge connecting vertices $i$ and $j$ if $\{i,j\} \in \mathcal{B}$.
\end{definition}

Note that the number of edges of this graph is $\dim V$.

\begin{Remark}
Denoting the complement of $\mathcal{G}(\mathcal{B})$ as $\mathcal{G}(\mathcal{B})^c$, it is worth noting that there is a natural correspondence between the set of triangles in $\mathcal{G}(\mathcal{B})^c$ and the set $\{\{i,j,k\} \in T \mid \{i,j\} \notin \mathcal{B},\{i,k\} \notin \mathcal{B},\{j,k\} \notin \mathcal{B} \}$
where ${T=\{\{a,b,c\} \mid 1 \leq a,b,c \leq n, c \notin \{ a , b \}, a \neq b \}}$. We will use this correspondence in the proof of Proposition \ref{Estimatesize}.
\end{Remark}

We will say that the set $\mathcal{B}\subseteq \mathcal{Y}$ constructed in Section \ref{Construction_Basis} is a \textit{tree of height one} if and only if $i\in \{1,\dots,n\}$ such that $i \in \{j,k\}$, for every $\{j,k\}\in \mathcal{B}$. The set $\mathcal{B}$ is a tree of height one if and only if graph $\mathcal{G}(\mathcal{B})$ is a tree with height one. Such a set $\mathcal{B}$ can be visualized  as follows:

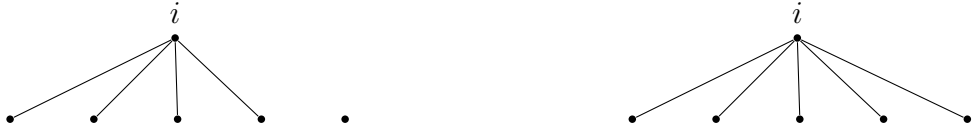
\[\begin{tikzpicture}[roundnode/.style={draw,shape=circle,fill=blue,minimum size=1mm}]

 \node[circle,fill,inner sep=1pt, label={north:$i$}]      (u1)                     {};
  
 \node[circle,fill,inner sep=1pt]      (u3)       [below left =of u1] {};
 \node[circle,fill,inner sep=1pt]      (u2)       [left=of u3] {};
  \node[circle,fill,inner sep=1pt]      (u4)       [right=of u3] {};
   \node[circle,fill,inner sep=1pt]      (u5)       [right=of u4] {};
   \node[circle,fill,inner sep=1pt]      (u6)       [right=of u5] {};

           \draw[-] (u1) -- (u2);
           \draw[-] (u1) -- (u3);
           \draw[-] (u1) -- (u4);
           \draw[-] (u1) -- (u5);
\end{tikzpicture}\qquad\qquad\qquad\qquad\quad
\begin{tikzpicture}[roundnode/.style={draw,shape=circle,fill=blue,minimum size=1mm}]
 \node[circle,fill,inner sep=1pt, label={north:$i$}]      (u1)                     {};
  
 \node[circle,fill,inner sep=1pt]      (u3)       [below left =of u1] {};
 \node[circle,fill,inner sep=1pt]      (u2)       [left=of u3] {};
  \node[circle,fill,inner sep=1pt]      (u4)       [right=of u3] {};
   \node[circle,fill,inner sep=1pt]      (u5)       [right=of u4] {};
   \node[circle,fill,inner sep=1pt]      (u6)       [right=of u5] {};

           \draw[-] (u1) -- (u2);
           \draw[-] (u1) -- (u3);
           \draw[-] (u1) -- (u4);
           \draw[-] (u1) -- (u5);
            \draw[-] (u1) -- (u6);
\end{tikzpicture} \]

\subsection{Some properties of a set associated with an alternating map and a special basis of $V$}  \vspace{.5em}
In the next lemma, we study some properties of the set $\mathcal{B}$. We will prove points $(i)$ and $(ii)$ and the remaining ones are left to the reader.

\begin{lemma} \label{Bproperties}
    The set $\mathcal{B}\subseteq \mathcal {Y}$ constructed in Section \ref{Construction_Basis} satisfies the following properties:
    \begin{enumerate}
        \item[(i)] The set $B=\{A(u_i,u_j)\in Y \mid \{i,j\}\in \mathcal{B}\}$ is a basis of $V$.
        \item[(ii)] If $\{r,s\}\in \mathcal{Y}$, then \label{spanningproperty}
        \begin{align*}
            A(u_r,u_s)\in \text{span } (\{ A(u_i,u_j)\in B\mid\ \{i,j\} \in \mathcal{B}, \{i,j\}\leq \{r,s\}\}).
        \end{align*}
         \item[(iii)] If $\{r,s\}\in \mathcal{Y}$, then 
         \begin{align*}
             \text{span}(\{A&(u_i,u_j)\in Y\mid\{i,j\}\leq\{r,s\}\})
             \\&= \text{span}(\{A(u_i,u_j)\in B\mid\{i,j\}\leq\{r,s\}\})\\
             &= \text{span}(\{A(u_i,u_j)\in B\mid\{i,j\}\in \mathcal{B}, \{i,j\}\leq\{r,s\}\}).
         \end{align*}
         \item[(iv)] If $\{r,s\}\in \mathcal{Y}$, then the set $\{\{i,j\}\in \mathcal{B}\mid\{i,j\}<\{r,s\}\}$ is either trivial or $\mathcal{B}_k$ for some $1\leq k\leq m$.
        
        \item[(v)] If $\{r,s\}\in \mathcal{B}$, then $A(u_r,u_s)\notin \text{span }(\{A(u_i,u_j)\in B\mid \{i,j\}<\{r,s\}\})$.  
    \end{enumerate}
\end{lemma}

\begin{proof}

\begin{enumerate}
    \item[\textit{(i)}] By construction, the set $B_m$ is linearly independent and hence a basis of $V$. Note that $B=B_m$ and hence 
    $\mathcal {B}$ satisfies the first property.

    \item[\textit{(ii)}] If $A(u_r,u_s) \in span(\{A(u_i,u_j) \in B \mid \{i,j\} \in \mathcal{B},\{i,j\} < \{r,s\} \})$, then by the construction of $\mathcal{B}$, we have $\{r,s\} \in \mathcal{B}$ and hence the proof.
    \end{enumerate}
\end{proof}
In the following proposition, we will show that $\mathcal{B}$ satisfies some additional properties, when $\dim V>2$ and $\dim U>\dim V+1$.
\begin{prop}\label{specialbasisd>k'+1}
    Let $U$ and $V $ be finite dimensional vector spaces over $\mathbb F$ with dimension $n$ and $m$, respectively. If $m>2$ and $n> m+1$, then there exists a basis $\{u_1,\ldots,u_{n}\}$ of $U$ with the order $u_1<\cdots< u_{n}$ such that $\mathcal{B}$ constructed in Section \ref{Construction_Basis} is either not a tree of height one, or else $\mathcal{B}$ satisfies the following properties:
    \begin{align}
        &\mathcal{B}=\{\{1,2\},\dots, \{1,{m+1}\}\}, \label{eq:p1} \\
        &A(u_1,u_i)=0 \text{ for all } m+2\leq i, \label{eq:p2}\\
        &A(u_i,u_j)=0 \text{ for all }i,j \text{ with } m+1<i<j, \label{eq:p3}\\
        &A(u_i,u_j)\in \text{span }\{A(u_1,u_i)\} \text{ for all } i,j  \text{ with } 2\leq i\leq m+1<j\leq n. \label{eq:p4} 
    \end{align} 
\end{prop}
\begin{proof}
We will divide the proof into four steps. 


\textbf{Step 1:} Let $X=\{u_1,\ldots,u_{n}\}$ be any basis of $U$ with order $u_1<\dots<u_{n}$, $\mathcal{B}$ and $B$ constructed as in Section \ref{Construction_Basis}. If $\mathcal{B}$ is not a tree of height one, then the result holds. So we can assume $\mathcal{B}$ is a tree of height one. Thus, there exists an $i_1\in \{1,\dots,n\}$ such that $\mathcal{B}=\{\{i_1,i_2\},\dots, \{i_1,i_{m+1}\}\}$. Let ${X\setminus\{u_{i_1},u_{i_2},\dots,u_{i_{m+1}}\}=\{w_{m+2},\dots, w_{n}\}}$, set $w_1=u_{i_1},\ldots, {w_{m+1}=u_{i_{m+1}}}$. The set $X_1=\{w_1,\dots,w_{n}\}$ is a basis of $U$, and let ${w_1<\cdots<w_{n}}$ be a total order on $X_1.$ Now we construct the set $\mathcal{B}^1$ and $B^1$ using the algorithm given in Section \ref{Construction_Basis}. Observe that $A(w_1,w_2)= A(u_{i_1},u_{i_2})$ and $A(w_1,w_3)=A(u_{i_1},u_{i_3})$ are linearly independent and hence $\{1,2\},\{1,3\}\in \mathcal{B}^1$. If $\mathcal{B}^1$ is not a tree of height one, then the result holds. If $\mathcal{B}^1$ is a tree of height one, then $\mathcal{B}^1=\{\{1,2\},\dots, \{1,m\}\}$ and we proceed to Step 2 with this $\mathcal{B}^1$.


    \textbf{Step 2:} 
    We can assume there exists a basis $X_1=\{u_1,\ldots,u_{n}\}$ of $U$ with order $u_1<\dots<u_{n}$ such that $\mathcal{B}^1$, constructed in Section \ref{Construction_Basis}, satisfies property \eqref{eq:p1}. Consider the linear map\begin{align*}
         L: U&\to V\\
         u&\mapsto A(u_1,u)
     \end{align*}
     Since $\{A(u_1,u_i)\mid 2\leq i\leq m+1\}$ is a basis of $V$, the map $L$ is surjective. Note that dim Ker $(L)= n-m$. Let $ \{u_1,w_{m+2},\dots, w_n\}$ be a basis of ${\text{Ker }(L)}$. Set $w_i=u_i$ for all $1\leq i\leq m+1$. Now note that $X_2= \{w_1,\dots,w_n\}$ is a basis of  $V$. Let $w_1<\dots< w_{n}$ be a total order on $X_2$. We construct $\mathcal{B}^2$ and 
 the basis $B^2$ of $V$ using the algorithm given in Section \ref{Construction_Basis}. It is easy to see that $\mathcal{B}^2=\{\{1,2\},\dots, \{1,m+1\}\}$ and  $A(w_1,w_r)=0$ for every $r\geq m+2,$ and we proceed to Step 3 with this $\mathcal{B}^2$. 
  
  \textbf{Step 3:} 
  We can assume there exists a basis $X_2 =\{u_1,\ldots,u_{n}\}$ of $U$ with order $u_1<\dots<u_{n}$ such that $\mathcal{B}^2$ constructed in Section \ref{Construction_Basis} satisfies properties \eqref{eq:p1} and \eqref{eq:p2}. Assume property \eqref{eq:p3} does not hold, that is,  ${A(u_r, u_s)\neq 0}$ for some $r,s$ with $n\geq s>r>m+1$. We will now show that there exists a basis $X_3=\{w_1,\ldots,w_n\}$ of $U$ with order ${w_1<\dots<w_n}$ such that the corresponding $\mathcal{B}^3$ is not a tree of height one.  Let ${j\in \{2,\dots, m+1\}}$ be such that $A(u_r,u_s) $ and $A(u_1,u_j)$ are linearly independent in $V,$ and let ${X_2\setminus\{u_r,u_s,u_1,u_j\}=\{w_5,\dots, w_{n}\}}$. Set $w_1=u_1,w_2=u_r,w_3=u_s,w_4=u_j$. Let $X_3= \{w_1,\ldots,w_n\},$ and let $w_1<\cdots<w_n$ be a total order on $X_3$. Now we construct the set $\mathcal{B}^3$ using the algorithm given in Section \ref{Construction_Basis}.
     We have $A(w_1,w_2)=A(u_1,u_r)=0$ and $A(w_1,w_3)=A(u_1,u_s)=0$. Moreover, $A(w_2,w_3)=A(u_r,u_s)$ and $A(w_1,w_4)=A(u_1,u_j)$ are linearly independent and hence $\{2,3\},\{1,4\}\in \mathcal{B}^3$. Thus, $\mathcal{B}^3$ is not a tree of height one, and the result holds. Therefore, we can assume that $\mathcal{B}^2$ has properties \eqref{eq:p1} through \eqref{eq:p3}, and we proceed to Step 4 with this $\mathcal{B}^2$. 
     
   \textbf{Step 4:} 
   We can assume there exists a basis $X_2 =\{u_1,\ldots,u_{n}\}$ of $U$ with order $u_1<\dots<u_{n}$ such that $\mathcal{B}^2$ constructed in Section \ref{Construction_Basis} satisfies properties \eqref{eq:p1} through \eqref{eq:p3}. Assume property \eqref{eq:p4} does not hold, that is, $A(u_r,u_s)\notin \text{span }\{ A(u_1,u_r) \}$ for all $r,s  \text{ with } {2\leq r\leq m+1<s\leq n}$. We will now show that there exists a basis $X_4=\{w_1,\ldots,w_n\}$ of $U$ with order $w_1<\dots<w_n$ such that corresponding $\mathcal{B}^4$ is not a tree of height one. Suppose there exists $r,s  \text{ with } 2\leq r\leq m+1<s\leq n$ such that ${A(u_r,u_s)\notin \text{span }\{ A(u_1,u_r) \}}$. Since $\mathcal{B}^2$ satisfies \eqref{eq:p1} and $m\geq 3$, there exists $2\leq j\leq m+1$ with $j\neq r$ such that $A(u_1,u_r),A(u_r,u_s)$ and $A(u_1,u_j)$ are linearly independent. Let $X_3\setminus\{u_1,u_r,u_s,u_j\}=\{w_5,\dots, w_n \}$. Set ${w_1=u_1,w_2=u_r,w_3=u_s,w_4=u_j}$,  and let $X_4= \{w_1,\ldots,w_{n} \}$. We give the total order ${w_1<\cdots<w_{n}}$ on $X_4$. Let $\mathcal{B}^4$  and the basis $B^4$ of $V$ be as constructed in Section \ref{Construction_Basis}. Since $s>m+1$, we have ${A(w_1,w_3)=A(u_1,u_s)=0}$, by \eqref{eq:p2}. Note that the elements $A(w_1,w_2)=A(u_1,u_r)$, ${A(w_2,w_3)=A(u_r,u_s)}$  and ${A(w_1,w_4)=A(u_1,u_j)}$ are linearly independent, and hence we have that ${\{1,2\},\{2,3\},\{1,4\}\in \mathcal{B}^4}$. Therefore, $\mathcal{B}^4$ is not a tree of height one, and the result holds. Thus, by this four-step process, we have produced a basis $X =\{u_1,\ldots,u_{n}\}$ of $U$ with the order $u_1<\dots<u_{n}$ such that $\mathcal{B}$ constructed in Section \ref{Construction_Basis} is either not a tree of height one, or else satisfies properties \eqref{eq:p1} through \eqref{eq:p4}.
\end{proof}

\subsection{A set of linearly independent elements in the image of $\Psi$} \label{Linear_Independence_Of_Psi}  \vspace{.5em}
For the rest of this section, we will adhere to the following assumptions: $U$ and $V $ be finite dimensional vector spaces over $\mathbb F$, and let $X=\{u_1,\ldots,u_n\} $ be a basis of $U$ with total order $u_1<\ldots<u_n.$ Let $A:U\times U\to V$ be an alternating map such that image of $A$ spans $V$. 

Consider the map,
\begin{align*}
        \Psi  : U \times U \times U &\to V \otimes U \\
	(x,y,z) & \mapsto A(x,y) \otimes z+ A(y,z) \otimes x+A(z,x) \otimes y.
    \end{align*}
    The following lemma lists some properties of $\Psi$, which can be verified easily and hence we omit the proof of this lemma.
    \begin{lemma} \label{Psiproperties}
        Let $\sigma$ be a permutation on $\{x,y,z\}$. We have the following properties for $\Psi$:
        \begin{enumerate}
            \item[(i)] $\Psi(x,x,y)=\Psi(x,y,x)=\Psi(y,x,x)=0$,
            \item[(ii)]$\Psi(\sigma(x),\sigma(y),\sigma(z))=sgn(\sigma)\Psi(x,y,z)$,
            \item[(iii)]$\Psi(x_1+x_2,y,z)=\Psi(x_1,y,z)+\Psi(x_2,y,z)$.
        \end{enumerate}
    \end{lemma}
Observe that this trilinear map induces the following linear map,
    \[ \Psi  : U \otimes U \otimes U \to V \otimes U.\]
Consider the following set,
\begin{align}
\mathcal{W}(\mathcal{B})&=\{\{a,b,c\} \mid 1 \leq a,b,c \leq n,  \{a,b\} \in \mathcal{B}, c \notin \{a,b\}\}. \label{eq:scriptW(scriptB)}
\end{align}

\begin{Remark}
    Note that $\{a,b,c\}\in \mathcal{W}(\mathcal{B})$ if and only if $1\leq a,b,c\leq n$, ${\{a,b\}\in \mathcal{B}}$ or  $\{b,c\}\in \mathcal{B}$ or  $\{c,a\}\in \mathcal{B}$, where  $a,b,c$ are pairwise distinct.
\end{Remark}
\begin{sloppypar}
    The total order on $X$ induces a total order on $\mathcal{W}(\mathcal{B})$ defined by ${\{a,b,c\} < \{i,j,k\}}$ if either     
\end{sloppypar} \begin{enumerate}
    \item $\max \{u_a,u_b,u_c\} < \max \{u_i,u_j,u_k\} $ or 
    \item $\max \{u_a,u_b,u_c\} = \max \{u_i,u_j,u_k\} =u_l$ and $\{a,b,c\} \setminus \{l \} < \{i,j,k\} \setminus \{l \}$,
\end{enumerate} 
where the comparison on sets of order two is as defined in Section \ref{Construction_Basis}.
The proof of the next lemma is easy and we omit it.
\begin{lemma} \label{tuplecomparison}
    Let $\{a,b,c\},\{i,j,k\} \in \mathcal{W}(\mathcal{B})$. If $\{a,b,c\} < \{i,j,k\}$ and ${l \in \{a,b,c\} \cap \{i,j,k\}}$, then $\{a,b,c\} \setminus \{l\} < \{i,j,k\} \setminus \{l\}$.
\end{lemma}
It is well known that if $B$ is a basis for $V$, then a basis for $V\otimes U$ can be obtained by taking $B \otimes X=\{A(u_i,u_j) \otimes u_k \mid A(u_i,u_j) \in B, u_k \in X\}$. Thus, we have the decomposition,	 $V \otimes U=\bigoplus\limits_{A(u_i,u_j) \in B} \bigoplus\limits _{k=1}^n \langle A(u_i,u_j) \otimes u_k \rangle$. Hence, we have the projection maps $P_{i,j,k}: V \otimes U \to \langle A(u_i,u_j) \otimes u_k \rangle$. With this set up, we give the following lemma:
\begin{lemma}\label{(a,b,c)properties}
   Let $\mathcal{B}$ and $B$ be as defined in Lemma \ref{Bproperties}. Assume that ${\{a,b\} \in \mathcal{Y}}$, $\{i,j\} \in \mathcal{B}$, and $A(u_i,u_j) \in B$.
    The following statements hold:
    \begin{enumerate}
        \item [(i)] If $u_c \neq u_k$, then $P_{i,j,k}(A(u_a,u_b) \otimes u_c) = 0$.
        \item[(ii)] Suppose that $\{a,b,c\} \in \mathcal{W}(\mathcal{B})$. If $P_{i,j,k}(\Psi(u_a \otimes u_b \otimes u_c)) \neq 0$, then $u_k$ is equal to one of $u_a$, $u_b$, or $u_c$.\label{2ndcoordinateoftensor}
        \item[(iii)] If $\{a,b\}< \{i,j\}$, then $P_{i,j,k}(\Psi(u_{\sigma(a)} \otimes u_{\sigma(b)} \otimes u_{\sigma(k)}))=0$, where $\sigma$ is any permutation of $\{a,b,k\}$.\label{projectionzero}
        
    \end{enumerate}
\end{lemma}

\begin{proof}
The proofs of \textit{(i)} and \textit{(ii)} are easy and are left to the reader. We will prove \textit{(iii)}.
        Note that ${\Psi(u_{\sigma(a)} \otimes u_{\sigma(b)} \otimes u_{\sigma(k)})=\text{sgn}(\sigma) \Psi(u_a \otimes u_b \otimes u_k)}$ for all $\sigma\in S_3$, by Lemma \ref{Psiproperties}\textit{(ii)}. 
        Therefore, it is enough to prove that $P_{i,j,k}(\Psi(u_a \otimes u_b \otimes u_k))=0$. By Lemma \ref{Bproperties}\textit{(ii)} and Lemma \ref{Bproperties}\textit{(iii)}, we have 
        \[{A(u_a,u_b) \in \text{span } \bigl(\{ A(u_s,u_t)\in B \mid \ \{s,t\} \leq \{a,b\}\}\bigr)}.\] 
        Since $\{a,b\}<\{i,j\}$, we obtain $A(u_i,u_j)\notin  \{ A(u_s,u_t)\in B \mid \ \{s,t\} \leq \{a,b\}\}$ as a consequence of Lemma \ref{Bproperties}\textit{(v)}. Hence,  when $A(u_a,u_b)$ is expressed as a linear combination in terms of the basis elements $B$, the coefficient of $A(u_i,u_j)$ is zero. This implies that the coefficient of $A(u_i,u_j) \otimes u_k$ is zero when $A(u_a,u_b) \otimes u_k$ is expressed as a linear combination in terms of the basis elements $B \otimes X$. Thus, $P_{i,j,k}(A(u_a, u_b) \otimes u_k))=0$. By \textit{(i)}, ${P_{i,j,k}(A(u_b,u_k) \otimes u_a)=0}$ and $P_{i,j,k}(A(u_k,u_a) \otimes u_b)=0$. Therefore, we have that ${P_{i,j,k}(\Psi(u_a \otimes u_b \otimes u_k))=0}$.
\end{proof}

We define $f: \mathcal{W}(\mathcal{B}) \to V \otimes U$ given by $f(\{a,b,c\})= \Psi(u_i \otimes u_j \otimes u_k)$ where $u_k= \max \{u_a,u_b,u_c\}$, $u_i=\min \{u_a,u_b,u_c\}$ and $u_j=\{u_a,u_b,u_c\} \setminus \{u_i,u_k\}$.
The following set which depends on $\mathcal{B}$ will play an important role in the rest of the paper: 
\begin{align}
W(\mathcal{B})&= f(\mathcal{W}(\mathcal{B}))  \label{eq:W(B)}
\end{align}
An element of $W(\mathcal{B})$ is of the form $A(u_i,u_j) \otimes u_k+A(u_j,u_k) \otimes u_i+ A(u_k,u_i) \otimes u_j$ where $u_i< u_j < u_k$.
In the next proposition,  we will show that the set $W(\mathcal{B})$ is linearly independent in $V \otimes U$.

\begin{prop} \label{Linearindependence}
     Let $\mathcal{B}$ and $B$ be as defined in Lemma \ref{Bproperties}. The following statements hold 
    \begin{enumerate}
        
        \item[(i)] The elements $f(\{a,b,c\})$, where $\{a,b,c\} \in \mathcal{W}(\mathcal{B})$,  are linearly independent in $V \otimes U$.
        \item[(ii)] The map $f: \mathcal{W}(\mathcal{B}) \to W(\mathcal{B})$ is a bijection.
    \end{enumerate}
\end{prop}
        
\begin{proof}
    \begin{enumerate}
        \item[\textit{(i)}] Suppose $\lambda_{\{a,b,c\}}\in F$ such that 

    \begin{equation} \label{linearindepenceeqn}
        \sum\limits_{\{a,b,c\}\in \mathcal{W}(\mathcal{B})} \lambda_{\{a,b,c\}}f(\{a,b,c\})=0.    
    \end{equation}
    
    Our aim is to prove that $\lambda_{\{a,b,c\}}=0$ for all $\{a,b,c\} \in \mathcal{W}(\mathcal{B})$. If there is an $\{a,b,c\} \in \mathcal{W}(\mathcal{B})$ with $\lambda_{\{a,b,c\}} \neq 0$, then we can find the largest element  $\{r,s,t\}$ in $\mathcal{W}(\mathcal{B})$ with $\lambda_{\{r,s,t\}} \neq 0$. We will use the projection maps, $P_{i,j,k}: V \otimes U \to \langle A(u_i,u_j) \otimes u_k \rangle$, to show that ${\lambda_{\{r,s,t\}}=0}$, a contradiction. By  \eqref{eq:W(B)}, we have that ${f({\{r,s,t\}})=\Psi(u_a \otimes u_b \otimes u_c)}$, where ${u_c= \max \{u_r,u_s,u_t\}}, {u_a=\min \{u_r,u_s,u_t\}}$ and ${u_b=\{u_r,u_s,u_t\} \setminus~\{u_a,u_c\}}$. Since $\{r,s,t\} \in \mathcal{W}(\mathcal{B})$,  we have that $\{a,b\}\in \mathcal{B} $ or $\{b,c\}\in  \mathcal{B}$ or $\{a,c\}\in  \mathcal{B}$. We will first consider the case $\{a,b\} \in \mathcal{B}$. 
          Applying the projection map ${P_{a,b,c}: V \otimes U \to \langle A(u_a,u_b) \otimes u_c \rangle}$ to \eqref{linearindepenceeqn}, we obtain
          \begin{equation} \label{projectedeqn1}
              P_{a,b,c}\Bigl (\sum\limits_{\{i,j,k\}\in \mathcal{W}(\mathcal{B})} \lambda_{\{i,j,k\}}f(\{i,j,k\})\Bigr)=0.
          \end{equation}
          Now we show that $P_{a,b,c}(\lambda_{\{i,j,k\}}f(\{i,j,k\}))=0$ when ${\{i,j,k\} \neq \{a,b,c\}}$.  By Lemma \ref{(a,b,c)properties}\textit{(ii)}, $P_{a,b,c}(f(\{i,j,k\}))=0$ for all $\{i,j,k\} \in \mathcal{W}(\mathcal{B})$ with ${c \notin \{i,j,k\}}$. Thus, we only need to consider elements $f(\{i,j,c\})$.  As $\{r,s,t\}$ is the largest element in $\mathcal{W}(\mathcal{B})$ such that $\lambda_{\{r,s,t\}}\neq 0$, we obtain that $\lambda_{\{i,j,c\}}=0$ for all $\{i,j,c\}$ with $\{i,j,c\} > \{r,s,t\}$. Therefore, ${P_{a,b,c}(\lambda_{\{i,j,c\}}f(\{i,j,c\}))=0}$ for all $\{i,j,c\}$ with ${\{i,j,c\} > \{r,s,t\}}$. If ${\{i,j,c\}<\{a,b,c\}}$, then $\{i,j\} <\{a,b\}$, using Lemma \ref{tuplecomparison}. Applying Lemma \ref{(a,b,c)properties}\textit{(iii)}, we get that ${P_{a,b,c}(f(\{i,j,c\}))=0}$. Thus, \eqref{projectedeqn1} yields
          \[ P_{a,b,c}\Bigl(\sum\limits_{\{i,j,k\}\in \mathcal{W}(\mathcal{B})} \lambda_{\{i,j,k\}}f(\{i,j,k\})\Bigr)=\lambda_{\{r,s,t\}}A(u_a,u_b)\otimes u_c=0,\]
          whence $\lambda_{\{r,s,t\}}=0$. The cases $\{b,c\} \in \mathcal{B}$ and $\{a,c\} \in \mathcal{B}$ follow similarly.
        \item[\textit{(ii)}] To prove that the map $f: \mathcal{W}(\mathcal{B}) \to W(\mathcal{B})$ is bijective, we need to check only injectivity. The linear independence of $f(\{a,b,c\})$ where $\{a,b,c\} \in \mathcal{W}(\mathcal{B})$ shows that $f$ is injective. This completes the proof.
    \end{enumerate}
\end{proof}
The next example shows why $B$ is insufficient to obtain the required linearly independent set in the image of $\Psi$, highlighting the necessity of $\mathcal{B}$.

\textbf{Example 2 (Necessity of $\mathcal{B}$):}
    One may wonder whether the set
    \begin{gather*}
        W(B)= \{\Psi(u_a \otimes u_b \otimes u_c) \mid  1 \leq a, b , c \leq n,u_a<u_b<u_c, A(u_a,u_b) \in B \text{ or } \\ A(u_b,u_c) \in B 
        \text{ or } A(u_a,u_c) \in B \}
    \end{gather*}
    is linearly independent. It is clear that $W(\mathcal{B}) \subseteq W(B)$, and we will show that $W(B)$ is not necessarily linearly independent. Consider the $\mathbb{F}_p$ vector spaces $U$ and $V$. Let $X=\{u_1,u_2,u_3,u_4,u_5\}$ be a basis of $U$ and $\{v_1,v_2,v_3,v_4,v_5\}$ be a basis of $V$. Let $A: U \times U \to V$ be the alternating bilinear map defined in Example 1. Let $X=\{u_1, \ldots , u_5\}$ be given the order $u_1 < \cdots < u_5$. We already observed that $B=\{A(u_1,u_2),A(u_1,u_3),A(u_2,u_4),A(u_2,u_5)\}$ and $\mathcal{B}=\{\{1,2\},\{1,3\},\{2,4\},\{2,5\}\}$. By definition,
   \begin{align*}
     W(B)=\{& \Psi(u_1 \otimes u_2 \otimes u_3),\Psi(u_1 \otimes u_2 \otimes u_4),\Psi(u_1 \otimes u_2 \otimes u_5),\Psi(u_1 \otimes u_3 \otimes u_4),\\
     & \Psi(u_1 \otimes u_3 \otimes u_5), \Psi(u_2 \otimes u_3 \otimes u_4),\Psi(u_2 \otimes u_4 \otimes u_5),\Psi(u_2 \otimes u_3 \otimes u_5), \\
     &\Psi(u_1 \otimes u_4 \otimes u_5),\Psi(u_3 \otimes u_4 \otimes u_5)\}. \\
     W(\mathcal{B})=\{& \Psi(u_1 \otimes u_2 \otimes u_3),\Psi(u_1 \otimes u_2 \otimes u_4),\Psi(u_1 \otimes u_2 \otimes u_5),\Psi(u_1 \otimes u_3 \otimes u_4), \\ &\Psi(u_1 \otimes u_3 \otimes u_5), \Psi(u_2 \otimes u_3 \otimes u_4), \Psi(u_2 \otimes u_4 \otimes u_5),\Psi(u_2 \otimes u_3 \otimes u_5)\}.
 \end{align*}
  From the elements listed above, note that $\Psi(u_1 \otimes u_4 \otimes u_5) \in W(B)$ while $\Psi(u_1 \otimes u_4 \otimes u_5) \notin W(\mathcal{B})$, and therefore we have that $W(\mathcal{B}) \subsetneq W(B)$.
  Note that $\Psi(u_1 \otimes u_4 \otimes u_5) \in W(B)$ because ${A(u_4,u_5)=A(u_2,u_4) \in B}$. Since $\{1,4\},\{4,5\},\{1,5\} \notin \mathcal{B}$, we conclude that ${\Psi(u_1 \otimes u_4 \otimes u_5) \notin W(\mathcal{B})}$. It is easy to observe that 
  ${\Psi(u_1 \otimes u_2 \otimes u_4)+\Psi(u_1 \otimes u_3 \otimes u_4)=\Psi(u_1 \otimes u_4 \otimes u_5)},$
  and hence $W(B)$ is not linearly independent. 
 

In the next example, we will show that given an arbitrary $\mathcal{S} \subseteq \mathcal{Y}$ such that $S=\{A(u_i,u_j)\in Y\mid \{i,j\}\in \mathcal{S}\}$ is a basis of $V$, $W(\mathcal{S})$ may not be linearly independent. 

\textbf{Example 3 (Importance of the method of construction of $\mathcal{B}$):} The construction of $\mathcal{B}$ is also crucial to the linear independence of $W(\mathcal{B})$ in Proposition \ref{Linearindependence}. Consider the $\mathbb{F}_p$ vector spaces $U$ and $V$. Let ${X=\{u_1,u_2,u_3,u_4,u_5\}}$ be a basis of $U$ and $\{v_1,v_2,v_3,v_4,v_5\}$ be a basis of $V$. Let ${A: U \times U \to V}$ be an alternating  bilinear map defined as follows,
    \begin{align*}
        A(u_1,u_2)&=v_1 \\
        A(u_1,u_3)&=v_2 & A(u_2,u_3)&=0  \\
        A(u_1,u_4)&=v_1+v_2 & A(u_2,u_4)&=v_3 & A(u_3,u_4)=0\\
        A(u_1,u_5)&=0 & A(u_2,u_5)&=v_5-v_4 & A(u_3,u_5)=v_4 & & A(u_4,u_5)=v_5.
    \end{align*}
  Let $X=\{u_1, \ldots , u_5\}$ be given the order $u_1 < \cdots < u_5$. Note that we have ${\mathcal{S}=\{\{1,2\},\{1,3\},\{2,4\}, \{3,5\},\{4,5\}\} \subseteq \mathcal{Y}}$ such that the corresponding $S=\{A(u_1,u_2),A(u_1,u_3), A(u_2,u_4),A(u_3,u_5),A(u_4,u_5)\}$ is a basis of $V$. Observe that ${\Psi(u_1 \otimes u_2 \otimes u_5),\Psi(u_1 \otimes u_3 \otimes u_5), \Psi(u_1 \otimes u_4 \otimes u_5) \in W(\mathcal{S})}$ and ${\Psi(u_1 \otimes u_2 \otimes u_5)+\Psi(u_1 \otimes u_3 \otimes u_5)=\Psi(u_1 \otimes u_4 \otimes u_5)}$. Hence, $W(\mathcal{S})$ is not linearly independent. Note that the construction given in Section \ref{Construction_Basis} yields ${\mathcal{B}=\{\{1,2\},\{1,3\},\{2,4\},\{2,5\},\{3,5\}\}}$ and the corresponding basis $B=\{A(u_1,u_2),A(u_1,u_3),A(u_2,u_4), A(u_2,u_5),A(u_3,u_5)\}$.

        \subsection{Estimating the size of $W(\mathcal{B})$}  \vspace{.5em}
In this section, our aim is to estimate a lower bound for $\dim\im\Psi.$ Towards this aim, we state our next proposition.
        \begin{prop}\label{Estimatesize}
            Let $U$ and $V$ be finite-dimensional $F$ vector spaces with $\dim U=n$ and $\dim V=m$. If $\mathcal{B}$ is the set constructed in Section \ref{Construction_Basis}, then the number of linearly independent elements in $W(\mathcal{B})$ is at least $ {n \choose 3}-{r \choose 3}- \binom{t}{2}$, where $r$ and $t$ are non-negative integers satisfying ${n \choose 2}-m ={r \choose 2} + t$ with $0 \leq t < r$. In particular, $\dim \operatorname{Im}{\Psi}\geq  {n \choose 3}-{r \choose 3}- {t\choose 2}.$
        \end{prop}
        \begin{proof}
            By Proposition \ref{Linearindependence}, we have that $W(\mathcal{B})$ is linearly independent in $V \otimes U$, and $|\mathcal{W}(\mathcal{B})|=|W(\mathcal{B})|$. Thus, it is enough to estimate the size of $\mathcal{W}(\mathcal{B})$. Since $\{a,b\} \in \mathcal{B}$ implies that $a \neq b$, we obtain that $\mathcal{W}(\mathcal{B})$ is a subset of $ {T=\{\{a,b,c\} \mid 1 \leq a,b,c \leq n, c \notin \{a,b\}, a \neq b\}}$. Noting that $T$ has size $\binom{n}{3}$, we get that $|\mathcal{W}(\mathcal{B})^c| = {n \choose 3}-|\mathcal{W}(\mathcal{B})|$. Thus, an upper bound on the size of $\mathcal{W}(\mathcal{B})^c$ gives a lower bound on the size of $\mathcal{W}(\mathcal{B})$. Observe that 
	 	\begin{align*}
   \mathcal{W}(\mathcal{B})^c= \{\{a,b,c\} \in T \mid \{a,b\} \notin \mathcal{B},\{a,c\} \notin \mathcal{B},\{b,c\} \notin \mathcal{B} \}. 
	 	\end{align*}
	 	To estimate the size of $\mathcal{W}(\mathcal{B})^c$, we consider the complement of the graph given in Definition \ref{graphofalternatingmap}. The number of edges of this graph is ${n \choose 2}-m$. Observe that $|\mathcal{W}(\mathcal{B})^c| $ is the number of triangles in this graph. By \cite[Theorem 2.5]{MavTho2023}, the number of triangles in the graph is at most ${r \choose 3}+\mybinom[.55]{{n \choose 2}-m-{r \choose 2}}{2}={r \choose 3}+\binom{t}{2}$. Thus, ${n \choose 3}-{r \choose 3}- \binom{t}{2}$ is a lower bound for the size of $\mathcal{W}(\mathcal{B})$.    
        \end{proof}
   The next proposition shows that the above estimate can be improved when $\mathcal{B}$ is not a tree of height one.
        \begin{prop}\label{Nonndentcase}
      Let $U$ and $V$ be finite-dimensional $F$ vector spaces with $\dim U=n$ and $\dim V=m$. If the set $\mathcal{B}$ constructed in Section \ref{Construction_Basis} is not a tree of height one, then $\vert\mathcal{W}(\mathcal{B})\vert\geq \Bigl(\sum\limits_{i=2}^{m+1}(n-i)\Bigr) + (m-2) $. In particular, $\dim \operatorname{Im}\Psi\geq \Bigl(\sum\limits_{i=2}^{m+1}(n-i)\Bigr) + (m-2)$.
\end{prop}
\begin{proof}

 Let $\{i_1,i_2\}$ be any element in $\mathcal{B}$. We consider two cases:
    \par \textbf{Case 1}: For any $\{k,l\}\in \mathcal B$,  $\{k,l\}\cap \{i_1,i_2\}\neq\varnothing$.

    
    
    For any $\{k,l\}\in \mathcal{B}$, we have that $i_1\in \{k,l\}$ or $i_2\in \{k,l\}$. Therefore, we obtain ${\mathcal B=\{\{i_1,i_2\}, \{i_1,a_1\},\dots, \{i_1,a_r\}, \{i_2,b_1\}, \dots,\{i_2,b_s\}\}}$ for some ${a_i,b_j\in \{1,\dots,n\}}$, where $1 \leq i \leq r$ and $1 \leq j \leq s$. 
    Since $\vert \mathcal{B}\vert=m$, we have $r+s+1=m$. Since $\mathcal{B}$ is not a tree of height one, we have $r\geq 1$ and $s\geq 1$. Define  $\mathcal{W}_0=\{\{i_1,i_2,c\}\mid c\in \{1,\dots,n\}\setminus\{i_1,i_2\}\}$. For $1\leq i\leq r$, define $\mathcal{W}_i^1=\{\{i_1,a_i,c\}\mid c\in \{1,\dots,n\}\setminus\{i_1,i_2,a_1,\dots ,a_i\}\}$. Similarly for $1\leq j\leq s$,  define $\mathcal{W}_j^2=\{\{i_2,b_j,c\}\mid c\in \{1,\dots,n\}\setminus\{i_1,i_2,b_1,\dots, b_j\}\}$. The sets $\mathcal{W}_0,\mathcal{W}^1_i,\mathcal{W}_j^2$ are contained in $\mathcal{W}(\mathcal{B})$, and $\vert\mathcal{W}_0\vert= n-2$, $\vert\mathcal{W}_i^1\vert= n-(i+2)$ and $\vert\mathcal{W}_j^2\vert= n-(j+2)$.
    Moreover, they are pairwise disjoint. 
    Thus, we have \begin{align*}
        \vert \mathcal{W}(\mathcal{B}) \vert&\geq\vert \mathcal{W}_0\vert +\sum\limits_{i=1}^r\vert \mathcal{W}_i^1\vert+\sum\limits_{i=1}^s \vert \mathcal{W}_i^2\vert 
     =(n-2)+ \sum\limits_{i=3}^{r+2}(n-i)+\sum\limits_{i=3}^{s+2}(n-i)\\
    &=(n-2)+ \sum\limits_{i=3}^{r+2}(n-i)+(n-3)+\sum\limits_{i=1}^{s-1}(n-(3+i)).
    \end{align*}
    Since $r\geq 1,$ we have $n-(3+i)\geq n-(r+2+i)$. Therefore
\begin{align*}
       \vert \mathcal{W}(\mathcal{B}) \vert&\geq (n-2)+ \sum\limits_{i=3}^{r+2}(n-i)+(n-3)+\sum\limits_{i=1}^{s-1}(n-(r+2+i))\\
        &=(n-2)+(n-3)+\sum\limits_{i=3}^{r+2+(s-1)}(n-i) \\
        &= (n-2)+(n-(m+1))+(m-2)+\sum\limits_{i=3}^{r+s+1}(n-i)   \\   
        &=\sum\limits_{i=2}^{r+s+1}(n-i)+ (n-(m+1))+(m-2).
    \end{align*}
    Since $r+s+1=m$, the proof is now complete.
    
    \par\textbf{Case 2:} There exists $\{r,s\}\in \mathcal B$, with  $\{r,s\}\cap \{i_1,i_2\}=\varnothing$.
    
    Let $\mathcal{B}= \{\{i_1^j,i_2^j\}\mid 1\leq j\leq m\}$. Without loss of generality, we may assume $i_1^1=i_1,\ i_2^1=i_2$ and $i_1^2=r,\ i_2^2=s$. For each $1\leq j\leq m$, our aim is to define $\mathcal{W}_j\subseteq \{\{i_1^j,i_2^j,c\}\mid c\in \{1,\dots,n\}\setminus\{i_1^j,i_2^j\}\}$ such that $\mathcal{W}_1,\dots,\mathcal{W}_j$ are contained in $ \mathcal{W}(\mathcal{B})$ and they are pairwise disjoint. Define 
    \begin{align*}
        \mathcal{W}_1&=\{\{i_1^1,i_2^1,c\}\mid c\in \{1,\dots,n\}\setminus\{i_1^1,i_2^1\}\} \text{ and }\\
        \mathcal{W}_2& =\{\{i_1^2,i_2^2,c\}\mid c\in \{1,\dots,n\}\setminus\{i_1^2,i_2^2\}\}.
    \end{align*}
    Note that $\vert \mathcal{W}_1\vert=\vert \mathcal{W}_2\vert = n-2$, and that they are pairwise disjoint. We will now inductively define $\mathcal{W}_j$ for $3\leq j\leq m$. Observe that if $\{i_1^j,i_2^j,c\}\in \mathcal{W}_i$ for some $i<j$, then $\{i_1^j,c\}=\{i_1^i,i_2^i\}$ or $\{i_2^j,c\}=\{i_1^i,i_2^i\}$.  We define
    \begin{gather*}
        C_j=\Bigl\{c\in \{1,\dots,n\}\setminus\{i_1^j,i_2^j\}\mid \{i_1^j,c\}\neq\{i_1^i,i_2^i\}  \text{ and  }\\\{i_2^j,c\}\neq \{i_1^i,i_2^i\} \text{ for any } i< j\Bigr\}    
    \end{gather*}
    and  $\mathcal{W}_j=\{\{i_1^j,i_2^j,c\}\mid c\in C_j\}$. We now prove that $\mathcal{W}_1,\dots, \mathcal{W}_j$ are pairwise disjoint. By induction, we have that $\mathcal{W}_1,\dots, \mathcal{W}_{j-1}$ are pairwise disjoint. If ${\{i_1^j,i_2^j,c\}\in \mathcal{W}_i \cap \mathcal{W}_j}$ for some $i<j$, then $\{i_1^j,c\}=\{i_1^i, i_2^i\}$ or ${\{i_2^j,c\}=\{i_1^i, i_2^i\}}$, a contradiction as $c\in C_j.$ Thus, $\mathcal{W}_i\cap \mathcal{W}_j=\varnothing$ for all $i<j$. Hence, $\mathcal{W}_1,\dots, \mathcal{W}_m$ are pairwise disjoint. Our aim is to estimate the size of $\mathcal{W}_j$ for $j\geq 3$. First note that $\vert \mathcal{W}_j\vert=\vert C_j\vert $, and hence it is enough to estimate the size of $C_j$. In order to obtain a lower bound for the size of $C_j$, we will give an upper bound on the size of the complement of $C_j$ in $\{1,\dots,n\}\setminus\{i_1^j,i_2^j\}$. Note that if $c\notin C_j$, then $\{i_1^j,c\}=\{i_1^i,i_2^i\}$ or $\{i_2^j,c\}=\{i_1^i,i_2^i\}$ for some $1\leq i<j$. Now we define $D=\{i_1^1,i_2^1\}\cup \{i_1^2,i_2^2\}$, and consider the following cases:   
  \par \textit{\textbf{Case 2a:}} $\{i_1^j,i_2^j\}\cap D=\varnothing$. 
  
  For any $c\in \{1,\dots,n\}\setminus\{i_1^j,i_2^j\}$, we have $\{i_1^j,c\}\neq \{i_1^1,i_2^1\}$, ${\{i_2^j,c\}\neq \{i_1^1,i_2^1\}}$, $\{i_1^j,c\}\neq \{i_1^2,i_2^2\}$ and $\{i_2^j,c\}\neq \{i_1^2,i_2^2\}$. Thus, if $c\notin C_j$, then $\{i_1^j,c\}=\{i_1^i,i_2^i\}$ or $\{i_2^j,c\}=\{i_1^i,i_2^i\}$ for some $3\leq i<j$. We will show that for any $i$ with $3\leq i< j $, there exist at most one $c\in \{1,\dots,n\}\setminus\{i_1^j,i_2^j\}$ such that  $\{i_1^j,c\}=\{i_1^i,i_2^i\}$ or $\{i_2^j,c\}=\{i_1^i,i_2^i\}.$ Towards that end, assume there exits $c, c'\in \{1,\dots,n\}\setminus\{i_1^j,i_2^j\}$ with $c\neq c'$ such that $\{i_r^j,c\}=\{i_1^i,i_2^i\}$
 and $\{i_s^j,c'\}=\{i_1^i,i_2^i\}$ for some $r,s\in\{1,2\}$. This implies the set $\{i_r^j,i_s^j,c,c'\}$ whose cardinality is at least three is a subset of $\{i_1^i,i_2^i\}$, a contradiction. Thus, the size of the complement of $C_j$ in $\{1, \ldots, n\}\setminus\{i_1^j,i_2^j\}$ is at most $j-3$, and therefore $\vert \mathcal{W}_j\vert = \vert C_j\vert\geq (n-2)-(j-3)=n-j+1 $.

\par \textit{\textbf{Case 2b:}} $\vert \{i_1^j,i_2^j\}\cap D\vert=1$.

Assume for some $r\in \{1,2\}$, $i_r^j\in  D$ and $i_s^j\notin D$, where ${s\in \{1,2\}\setminus\{r\}}$. Since $\{i_1^1,i_2^1\}\cap\{i_1^2,i_2^2\}=\varnothing$, we have $\{i_1^j, i_2^j\} \cap \{i_1^t,i_2^t\}=\varnothing$ for some ${t\in \{1,2\}}$. Therefore, if $c\notin C_j$, then $\{i_1^j,c\}=\{i_1^i,i_2^i\}$ or $\{i_2^j,c\}=\{i_1^i,i_2^i\}$ for some ${i\in \{1,\dots, j-1\}\setminus \{t\}}$.  Moreover, for any such $i$, there exist at most one $c\in \{1,\dots,n\}\setminus\{i_1^j,i_2^j\}$ such that  $\{i_1^j,c\}=\{i_1^i,i_2^i\}$ or $\{i_2^j,c\}=\{i_1^i,i_2^i\}$. Thus, the size of the complement of $C_j$ in $\{1, \ldots, n\}\setminus\{i_1^j,i_2^j\}$ is at most $j-2$, and therefore $\vert \mathcal{W}_j\vert = \vert C_j\vert\geq (n-2)-(j-2)=n -j $.



    
    \par \textit{\textbf{Case 2c:}}  $\{i_1^j,i_2^j\}\subseteq D$.
    
    We will first consider the case $j\geq 3$ being the smallest integer such that $\{i_1^j,i_2^j\}\subseteq D$. If $c\notin C_j$, then $\{i_1^j,c\}=\{i_1^i,i_2^i\}$ or $\{i_2^j,c\}=\{i_1^i,i_2^i\}$ for some ${i\in \{1,\dots, j-1\}}$.  Thus, we have that ${\vert \mathcal{W}_j\vert \geq (n-2)-(j-1)= n-j-1 }$. Now we consider the case where $j$ is not the smallest integer such that ${\{i_1^j,i_2^j\}\subseteq D}$. Thus, there exists $j'\geq 3$ with $j'<j$ such that $\{i_1^{j'},i_2^{j'}\}\subseteq D$. Note that  ${\{i_1^{j'},i_2^{j'}\}=  \{i_1^j,i_2^j\}}$ is not possible. If $\{i_1^{j'},i_2^{j'}\}\cap  \{i_1^j,i_2^j\}=\varnothing$, then ${\{i_1^j,c\}\neq\{i_1^{j'},i_2^{j'}\}}$ and ${\{i_2^j,c\}\neq\{i_1^{j'},i_2^{j'}\}}$ for any ${c\in \{1,\dots,n\}\setminus \{i_1^j,i_2^j\}}$. Therefore, if $c\notin C_j$, then $\{i_1^j,c\}=\{i_1^i,i_2^i\}$ or $\{i_2^j,c\}=\{i_1^i,i_2^i\}$ for some ${i\in \{1,\dots, j-1\}\setminus \{j'\}}$, and therefore ${\vert \mathcal{W}_j\vert \geq (n-2)-(j-2)=n-j}$. Assume ${\{i_1^{j'},i_2^{j'}\}\cap  \{i_1^j,i_2^j\}\neq\varnothing}$, and let $i_r^{j'}$ be the element in the intersection. Letting $s\in\{1,2\}\setminus \{r\}$, we observe that
    \begin{equation}\label{E:4.2.1}
        \{i_s^{j'},i_1^j\}\in\bigl\{\{i_1^1,i_2^1\},\{i_1^2,i_2^2\}\bigr\} \text{ or } \{i_s^{j'},i_2^j\}\in\bigl\{\{i_1^1,i_2^1\},\{i_1^2,i_2^2\}\bigr\}. 
    \end{equation}
If $c\notin C_j$, then $\{i_1^j,c\}=\{i_1^i,i_2^i\}$ or $\{i_2^j,c\}=\{i_1^i,i_2^i\}$ for some $1\leq i<j$. Moreover, as in \textit{Case 2a}, for any $i$ with $1\leq i< j $, there exist at most one $c\in \{1,\dots,n\}\setminus\{i_1^j,i_2^j\}$ such that  $\{i_1^j,c\}=\{i_1^i,i_2^i\}$ or $\{i_2^j,c\}=\{i_1^i,i_2^i\}.$ Note that $c=i_s^{j'}$ satisfies $\{i_1^j,c\}=\{i_1^{j'},i_2^{j'}\}$ or  $\{i_2^j,c\}=\{i_1^{j'},i_2^{j'}\}$. By \eqref{E:4.2.1}, $c=i_s^{j'}$ also satisfies $\{i_1^j,c\}=\{i_1^{t},i_2^{t}\}$ or  $\{i_2^j,c\}=\{i_1^{t},i_2^{t}\}$ for some $t\in\{1,2\}$. Therefore, the size of the complement of $C_j$ in $\{1, \ldots, n\}\setminus\{i_1^j,i_2^j\}$ is at most $j-2$, and hence $\vert \mathcal{W}_j\vert = \vert C_j\vert\geq (n-2)-(j-2)=n-j. $
Thus, $\vert \mathcal{W}_j\vert \geq n-j$ for $3\leq j\leq m$, except possibly for one $j$. For such a $j$, $\vert \mathcal{W}_j\vert \geq n-j-1.$ Therefore,
\begin{align*}
        \vert \mathcal{W}(\mathcal{B})\vert &\geq\sum\limits_{j=1}^{m}\vert \mathcal{W}_j\vert
        \geq (n-2)+(n-2)+\Bigl(\sum_{j=3}^{m}(n-j)\Bigr)-1 \\
        &= \Bigl(\sum\limits_{i=2}^{m+1}(n-i)\Bigr) + (m-2).
    \end{align*} 
\end{proof}
 	\section{Bounds on the size of the Schur multiplier of \textit{p}-groups}
      
       Let $G$ be finite $p$-group and let $X=\{g_1,\cdots,g_\delta\} \subseteq G$ be such that its image forms a minimal generating set of $\overline{G}:=G/Z$. It is easy to see that the simple commutators $\overline{[a,b]}$ where $a,b \in X$, forms a generating set for $\gamma_2G/\gamma_3G$. By abuse of notation, we shall omit the use of the bar notation, but it will be clear from the context. 
Consider the homomorphism defined in \cite[Proposition 1]{EllWie1999}:
\begin{align}
\label{eqn:Psi_2_definition}
\begin{split}
 \Psi_2 & : \overline{G}^{ab} \otimes \overline{G}^{ab} \otimes \overline{G}^{ab} \to \gamma_2G/\gamma_3G \otimes \overline{G}^{ab} 
\\
\overline{x} \otimes \overline{y} \otimes \overline{z} & \mapsto \overline{[x,y]} \otimes \overline{z}+ \overline{[y,z]} \otimes \overline{x}+\overline{[z,x]} \otimes \overline{y}.
\end{split}
\end{align}
  Note that we have the natural homomorphisms $P_1 : \gamma_2G/\gamma_3G \to \frac{\gamma_2G/\gamma_3G}{\Phi(\gamma_2G/\gamma_3G)}$ and $P_2: \overline{G}^{ab} \to \frac{\overline{G}^{ab}}{\Phi(\overline{G}^{ab})}$, which induces the surjective homomorphism 
  \begin{align*}
  P_1 \otimes P_2: \gamma_2G/\gamma_3G \otimes \overline{G}^{ab} \to \frac{\gamma_2G/\gamma_3G}{\Phi(\gamma_2G/\gamma_3G)} \otimes_{\mathbb{Z}} \frac{\overline{G}^{ab}}{\Phi(\overline{G}^{ab})}.
  \end{align*}
  The next lemma is standard and we omit its proof.

   \begin{lemma} \label{vectortensor}
            Let $R$ be a commutative ring with unit, $I$ be an ideal of $R$, and $M_1$, $M_2$ be $R$-modules. If $IM_1=IM_2=0$, then ${M_1 \otimes_{R} M_2 \cong M_1 \otimes_{R/I} M_2}$.
        \end{lemma}

   By Lemma \ref{vectortensor}, we get that 
   \begin{align*}
   \frac{\gamma_2G/\gamma_3G}{\Phi(\gamma_2G/\gamma_3G)} \otimes_{\mathbb{Z}} \frac{\overline{G}^{ab}}{\Phi(\overline{G}^{ab})} \cong \frac{\gamma_2G/\gamma_3G}{\Phi(\gamma_2G/\gamma_3G)} \otimes_{\mathbb{F}_p} \frac{\overline{G}^{ab}}{\Phi(\overline{G}^{ab})}.
   \end{align*}
   Set $U=\frac{\overline{G}^{ab}}{\Phi(\overline{G}^{ab})}$, $V=\frac{\gamma_2G/\gamma_3G}{\Phi(\gamma_2G/\gamma_3G)}$ and let $A: \frac{\overline{G}^{ab}}{\Phi(\overline{G}^{ab})} \times \frac{\overline{G}^{ab}}{\Phi(\overline{G}^{ab})} \to \frac{\gamma_2G/\gamma_3G}{\Phi(\gamma_2G/\gamma_3G)}$ be the commutator map. With this setup, let 
   \begin{align*}
   \Psi : \frac{\overline{G}^{ab}}{\Phi(\overline{G}^{ab})} \otimes \frac{\overline{G}^{ab}}{\Phi(\overline{G}^{ab})} \otimes \frac{\overline{G}^{ab}}{\Phi(\overline{G}^{ab})}  \to \frac{\gamma_2G/\gamma_3G}{\Phi(\gamma_2G/\gamma_3G)} \otimes \frac{\overline{G}^{ab}}{\Phi(\overline{G}^{ab})}
   \end{align*}
   be the map defined in Section \ref{Linear_Independence_Of_Psi}. The proof of the next result is easy and we omit it.

   \begin{lemma}\label{CommuativeDiagram}
       If $G$ be a finite $p$-group, then the following diagram is commutative,
\begin{equation*}
\begin{tikzcd}[row sep=huge]
\overline{G}^{ab} \otimes \overline{G}^{ab} \otimes \overline{G}^{ab} \arrow[r,"\Psi_2"] \arrow[d,swap,"P_2 \otimes P_2 \otimes P_2"] &
\gamma_2G/\gamma_3G \otimes \overline{G}^{ab} \arrow[d,"P_1 \otimes P_2"]   \\
\frac{\overline{G}^{ab}}{\Phi(\overline{G}^{ab})} \otimes \frac{\overline{G}^{ab}}{\Phi(\overline{G}^{ab})} \otimes \frac{\overline{G}^{ab}}{\Phi(\overline{G}^{ab})} \arrow[r,swap,"\Psi"]  &
\frac{\gamma_2G/\gamma_3G}{\Phi(\gamma_2G/\gamma_3G)} \otimes \frac{\overline{G}^{ab}}{\Phi(\overline{G}^{ab})}
\end{tikzcd}
\end{equation*} 
\end{lemma}


Now we come to our main theorem.

        \begin{theorem} \label{mainbound}
            Let $G$ be a $p$-group of order $p^n$. Assume that $d(G)=d$, $d(G/Z)=\delta$, $d(\gamma_2G/\gamma_3G)=k'$ and $|\gamma_2G|=p^k$. Let $r$ and $t$ be non-negative integers such that ${\delta \choose 2}-k'={r \choose 2}+t$, where $0 \leq t < r$. Then,
            \[|M(G)| \leq p^{\frac{1}{2}(d-1)(n+k)-k(d-\delta)-\binom{\delta}{3}+ \binom{r}{3} + \binom{t}{2}}.\]
        \end{theorem}
        \begin{proof}
             By the inequality given by Ellis and Wiegold \eqref{EllisInequality}, we have 
\begin{align*}
    |M(G)||\gamma_2G||\im{\Psi_2}| \leq |M(G^{ab})|\prod^c_{i=2}|\gamma_iG/\gamma_{i+1}G \otimes \overline{G}^{ab}|,
\end{align*}
  
where $c$ is the nilpotency class of $G$. It is easy to see that \[\prod^c_{i=2}|\gamma_iG/\gamma_{i+1}G \otimes \overline{G}^{ab}|\leq p^{k\delta}.\] Moreover, by \cite[Lemma 2.3]{NR12}, we have $|M(G_{ab})|\leq p^{\frac{1}{2}(n-k)(d-1)}.$ Thus,
             \begin{align}\label{schur_Psi_image_inequality}
                 |M(G)||\im{\Psi_2}| \leq p^{\frac{1}{2}(d-1)(n+k)-k(d-\delta)}.
             \end{align}
             Now we will show that $|\im{\Psi_2}| \geq p^{\binom{\delta}{3}- \binom{r}{3} - \binom{t}{2}}$. Let $X=\{g_1,\ldots,g_{\delta}\}$ be elements of $G$ such that its images $\overline{X}$ form a basis of $\frac{\overline{G}^{ab}}{\Phi(\overline{G}^{ab})}$. Let ${\overline{g_1}<\dots<\overline{g_{{\delta}}}}$ be a total order on $\overline{X}$ and let $\mathcal{B}$ be the set constructed in Section \ref{Construction_Basis}. Corresponding to this $\mathcal{B}$, we have $W(\mathcal{B})$ as defined in \eqref{eq:W(B)}. By Proposition \ref{Estimatesize}, we have that $W(\mathcal{B})$ has at least $\binom{\delta}{3}- \binom{r}{3} - \binom{ t }{2} $ linearly independent elements. By Lemma \ref{CommuativeDiagram}, we have $W(\mathcal{B}) \subseteq (P_1 \otimes P_2)(\im{\Psi_2})$. Hence, $(P_1 \otimes P_2)(\im{\Psi_2})$ has at least $p^{\binom{\delta}{3}- \binom{r}{3} - \binom{t }{2} }$ elements, and thus $|\im{\Psi_2}| \geq p^{\binom{\delta}{3}- \binom{r}{3} - \binom{ t }{2} }$. Plugging this estimate in \eqref{schur_Psi_image_inequality} gives the required bound.
        \end{proof}
        As a corollary, we obtain a bound on the size of the second cohomology group $H^2(G,\mathbb{Z}/p\mathbb{Z})$ with coefficient in $\mathbb{Z}/p\mathbb{Z}$.
        \begin{corollary}\label{main bound Zp coefficient }
                    Let $G$ be a $p$-group of order $p^n$. Assume that $d(G)=d$, $d(G/Z)=\delta$, $d(\gamma_2G/\gamma_3G)=k'$ and $|\gamma_2G|=p^k$. Let $r$ and $t$ be non-negative integers such that ${\delta \choose 2}-k'={r \choose 2}+t$, where $0 \leq t < r$. Then,
            \[|H^2(G,\mathbb{Z}/p\mathbb{Z})| \leq p^{\frac{1}{2}(d-1)(n+k)-k(d-\delta)-\binom{\delta}{3}+ \binom{r}{3} + \binom{t}{2}+ d }.\]
        \end{corollary}
\begin{proof}
      For any finite abelian group $A$, by \cite[Corollary 2.1.20]{Karp1987}, we have \[H^2(G,A)\cong [(G/\gamma_2G)\otimes A]\times[M(G)\otimes A].\]Since $d(G/\gamma_2G)=d(G)$, we note that $|(G/\gamma_2G)\otimes \mathbb{Z}/p\mathbb{Z}|= p^d$. Moreover, we have that $|M(G)\otimes \mathbb{Z}/p\mathbb{Z}|=p^{d(M(G))}\leq |M(G)|$. Now the result follows from Theorem \ref{mainbound}.
\end{proof}
The next theorem gives a bound on the size of the Schur multiplier for a finite group.
\begin{theorem}\label{general main bound}
    Let G be a finite group of order $n$ and let $G_p$ be a Sylow $p$-subgroup of $G$ of order $p^{n_p}$. Assume that $d(G_p)=d_p$, $d(G_p/Z(G_p))=\delta_p$, $d(\gamma_2G_p/\gamma_3G_p)=k'_p$ and $|\gamma_2G_p|=p^{k_p}$. Let $r_p$ and $t_p$ be non-negative integers such that ${\delta_p \choose 2}-k'_p={r_p \choose 2}+t_p$, where $0 \leq t_p < r_p$. Then,
            \[|M(G)| \leq \prod p^{\frac{1}{2}(d_p-1)(n_p+k_p)-k_p(d_p-\delta_p)-\binom{\delta_p}{3}+ \binom{r_p}{3} + \binom{t_p}{2}},\] 
            where the product runs over all the primes dividing $n$.
\end{theorem}
\begin{proof}
    As a consequence of Theorem 4, Chapter IX of \cite{JPS1979}, we have an injective map $M(G)\to \prod_p M(G_p)$. Now the proof follows from Theorem \ref{mainbound}. 
\end{proof}
   
The next lemma shows that Theorem \ref{mainbound} improves the existing bounds for the size of the Schur multiplier of $p$-groups.
        \begin{lemma}\label{BoundComparison}
            The bound on the size of the Schur multiplier in Theorem \ref{mainbound} improves the bound in \cite[Theorem 1.4]{Rai2024}.
        \end{lemma}
        \begin{proof}
        It is easy to see that the bound in \eqref{schurinequality} is stronger than the one in \eqref{eq:Raitheorembound}. Now we will show that the bound obtained in Theorem \ref{mainbound} improves the one in \eqref{schurinequality}. It is enough to show that ${\sum \limits_{i=2}^{\min (\delta,k'+1)}(\delta-i) \leq \binom{\delta}{3}- \binom{r}{3} - \binom{ t} {2}}$, where $t=\binom{\delta}{2} -k'- \binom{r}{2}$. The following binomial identity will be used throughout the proof:
            \begin{equation} \label{nchoosekidentity}
                {n \choose k}={n-1 \choose k}+{n-1 \choose k-1}.
           \end{equation}
        \par\textbf{Case 1:} $\delta < k'+1$.
        
        Note that $\sum \limits_{i=2}^{\min (\delta,k'+1)}(\delta-i)=\sum \limits_{i=2}^{\delta} (\delta-i)={\delta -1 \choose 2}$. We will show that ${{\delta -1 \choose 2} < \binom{\delta}{3}- \binom{r}{3} - \binom{ t}{2} }$. Towards that, we will first prove that 
            \begin{equation}\label{rclaim}
                t < r \leq \delta-2.
            \end{equation}
            By the hypothesis in Theorem \ref{mainbound}, we have that ${\delta \choose 2}-k' = {r \choose 2}+t$. By \eqref{nchoosekidentity}, we have that ${\delta \choose 2} - k' ={\delta-1 \choose 2} + \delta-1-k'$. Since $\delta < k'+1$, we have that $\delta-1-k' < 0$, and hence ${\delta \choose 2} - k' < {\delta-1 \choose 2}$. From this, we obtain that ${r \choose 2} + t < {\delta-1 \choose 2}$. This shows that $r \leq \delta-2$. The definition of $r$ and $t$ gives the inequality $t < r $, proving \eqref{rclaim}. By \eqref{rclaim}, we have that
            \begin{align*}
                \binom{\delta}{3}- \binom{r}{3} - \binom{ t }{2} >& \binom{\delta}{3}- \binom{\delta-2}{3} - \binom{\delta-2}{2} \\
                =&\binom{\delta}{3}- \binom{\delta-1}{3}=\binom{\delta-1}{2}.
            \end{align*} 
            \par\textbf{Case 2:} $\delta \geq k'+1$.
            
            Observe that ${\sum \limits_{i=2}^{\min (\delta,k'+1)}(\delta-i)=\sum \limits_{i=2}^{k'+1} (\delta-i)=\binom{\delta-1}{2}-\binom{ \delta-(1 + k')}{2}}$. We will show that \[\binom{\delta-1}{2}-\binom{ \delta-(1 + k')}{2}= \binom{\delta}{3}- \binom{r}{3} - \binom{t }{2}. \] 
            By the hypothesis in Theorem \ref{mainbound}, we have ${\delta \choose 2}-k' = {r \choose 2} + t$. Note that ${\delta \choose 2} - k' ={\delta-1 \choose 2} + \delta-1-k'$, by \eqref{nchoosekidentity}. Since $\delta \geq k'+1$, we obtain ${0 \leq \delta-1-k' < \delta-1}$. Therefore, ${\delta-1 \choose 2} \leq {\delta \choose 2} - k' < {\delta \choose 2}$, and hence $r=\delta-1$ and $t=\binom{\delta}{2} -k'- \binom{\delta-1}{2} $. Thus,
            \begin{align*}
                \binom{\delta}{3}- \binom{r}{3} - \binom{ t }{2} =& \binom{\delta}{3}- \binom{\delta-1}{3} - \binom{ \binom{\delta}{2} -k'- \binom{\delta-1}{2} }{2} \\
                =& \binom{\delta-1}{2}-\binom{ \delta-(1 + k')}{2}.
            \end{align*}
            Hence the proof.
        \end{proof}
        The next lemma follows from the computations in Lemma \ref{BoundComparison}, Case 2.

    \begin{lemma}\label{Lemma r and sum}
        Let $r$ and $t$ be non-negative integers such that ${\delta \choose 2}-k'={r \choose 2}+t$, where $0 \leq t < r$. If $\delta>k'+1$, then $\binom{\delta}{3}+ \binom{r}{3} + \binom{ t} {2} =\sum\limits_{i=1}^{k'+1}(\delta-i)$.
    \end{lemma}
        In the next theorem, we aim to improve the bound for special $p$-groups with $\delta > k'+1$.
          \begin{theorem}\label{mainboundd>k'+1}
            Let $G$ be a special $p$-group of order $p^n$ with $d(G)=d$ and $d(\gamma_2G)=k$, where $k>2$. If $d>k+1$, then
                \[|M(G)| \leq p^{\frac{1}{2}(d-1)(n+k)-\big(\sum\limits_{i=2}^{k+1}(d-i)\big)-(k-2) }.\]
        \end{theorem} 
        
\begin{proof}
    Note that for a special group $Z=\gamma_2G=\Phi(G)$, and hence we have that ${d(G)=d(G/Z)}$, $d(\gamma_2G)=k, \vert \gamma_2G\vert =p^{k}, \Phi(G/Z)=1$ and $\overline{
G}^{ab}=G/Z$. By Proposition \ref{specialbasisd>k'+1}, there exists a basis $\{g_1,\ldots,g_{d}\}$ of $G/Z$ with the order $g_1<\cdots< g_{d}$ such that the set $\mathcal{B}$ constructed in Proposition \ref{specialbasisd>k'+1} is either not a tree of height one, or else satisfies the following properties:
    \begin{align}
        &\mathcal{B}=\{\{1,2\},\dots, \{1,{k+1}\}\}, \label{Geq:p1} \\
        &[g_1,g_i]=1 \text{ for all $i$ with } k+2\leq i, \label{Geq:p2}\\
        &[g_i,g_j]=1 \text{ for all }i,j \text{ with } k+1<i<j, \label{Geq:p3}\\
        &[g_i,g_j]\in \text{span }\{[g_1,g_i]\} \text{ for all } i,j  \text{ with } 2\leq i\leq k+1<j\leq d. \label{Geq:p4} 
        \end{align}
        If $\mathcal{B}$ is not a tree of height one, then by Propositions \ref{Linearindependence} and \ref{Nonndentcase}, $W(\mathcal{B}) $ defined in \eqref{eq:W(B)} has at least ${\left(\sum\limits_{i=2}^{k+1}(d-i)\right) + (k-2)}$ linearly independent elements. Note that $W(\mathcal{B})\subseteq \im{\Psi_2},$ and hence the proof follows from \eqref{schur_Psi_image_inequality}. Thus, we can assume $\mathcal{B}$ satisfies \eqref{Geq:p1}, \eqref{Geq:p2}, \eqref{Geq:p3} and \eqref{Geq:p4}. For all $s$ with ${2\leq s\leq k+1}$, let $[g_s, g_{k+2}]= [g_1,g_s]^{\alpha_s}$ for some $\alpha_s\in \mathbb F_p$. We will show that there exist $s$ and $t$ such that $2\leq s,t\leq k+1$ and $\alpha_s\neq \alpha_{t}$. If $\alpha_s = \alpha_t= \alpha$ for every $s$ and $t$ with $2\leq s,t\leq k+1$, then for all $i$ with $1\leq i\leq d$, we have $[g_1^{\alpha}g_{k+2},g_i]=1$ by \eqref{Geq:p2} and \eqref{Geq:p3}. Thus, $g_1^{\alpha}g_{k+2}\in Z$, a contradiction. Thus, for each $i$ with ${2\leq i\leq k+1}$, choose a $c(i)\in\{s,t\}$ such that $\alpha_i\neq \alpha_{c(i)} $. Define $\mathcal{W}=\{\{i, c(i),k+2\}\mid 2\leq i\leq k+1, i\neq t\}$ and ${W=\{ \Psi_2(g_i\otimes g_{c(i)}\otimes g_{k+2})\mid 2\leq i\leq k+1 ,i\neq t\}}$. If $i,j$ are such that $2\leq i,j\leq k+1$, $i\neq j$ and $i \neq t \neq j$, then $\{i, c(i),k+2\}\neq \{j, c(j),k+2\}$. Note that $\mathcal{W}(\mathcal{B})$ defined in \eqref{eq:scriptW(scriptB)} can be written as $\mathcal{W}(\mathcal{B})=\bigsqcup\limits_{i=2}^{k+1} \mathcal{W}_i$, where $\mathcal {W}_i=\{\{1, i,c\}\mid i<c\leq d\}$. We now show that the set ${W(\mathcal{B})}\cup W$ is linearly independent. Let $\lambda_{\{a,b,c\}}\in\mathbb F_p$ be such that 
        \begin{equation}  \label{linearindepenceeqn2 d>k+1}
          \Bigl(\prod\limits_{i=2}^{k+1} \!\!\!\!\!\!\!\!\!\!\!\prod\limits_{\ \ \ \ \ \ \{1,i,c\}\in \mathcal{W}_i} \!\!\!\!\!\!\!\!\!\!\!\!\Psi_2(g_1\otimes g_{i}\otimes g_{c})^{\lambda_{\{1,i,c\}}}\Bigr) \Bigl(\!\!\!\!\!\!\!\!\!\!\!\!\!\!\! \!\prod\limits_{\ \ \ \ \ \ \ \ \{i,c(i),k+2\}\in \mathcal{W}}\!\!\!\!\!\!\!\!\!\!\!\!\! \! \! \! \! \! \! \Psi_2(g_i\otimes g_{c(i)}\otimes g_{k+2})^{\lambda_{\{i,c(i),k+2\}}}\Bigr)=1. 
      \end{equation}
        Our aim is to prove that $\lambda_{\{a,b,c\}}=0$ for all ${\{a,b,c\} \in \mathcal{W}(\mathcal{B})\cup \mathcal{W}}$. For all $j$ such that ${2\leq j\leq k+1}$, consider the projection maps 
        \[{P_{1,j,c(j)}: \gamma_2G/\gamma_3G \otimes G/Z \to \langle [g_1,g_j] \otimes g_{c(j)} \rangle}.\]  
        Let ${\{1,a,b\}\in \mathcal{W}(\mathcal{B})}$. If ${c(j)\notin \{1,a,b\}}$, then ${P_{1,j,c(j)}(\Psi_2(g_1\otimes g_a\otimes g_b))=1}$, by Lemma \ref{(a,b,c)properties}\textit{(ii)}. If ${j\notin \{1,a,b\}}$, then by \eqref{Geq:p1} and \eqref{Geq:p2}, the exponent of $[g_1,g_j]$  is $0$ when $[g_1,g_a]$ and $[g_1,g_b]$ are written as linear combinations in terms of basis ${B=\{[g_1,g_s]\mid 2\leq s\leq k+1\}}$. Therefore, we have that ${P_{1,j,c(j)}(\Psi_2(g_1\otimes g_a\otimes g_b))=1}$. Hence, if ${P_{1,j,c(j)}(\Psi_2(g_1\otimes g_a\otimes g_b))\neq 1}$ and $\{1,a,b\}\in \mathcal{W}(\mathcal{B})$, then ${\{1,a,b\}=\{1,j,c(j)\}}$. Thus, 
        \[
            P_{1,j,c(j)}\Bigl(  \prod\limits_{i=2}^{k+1} \prod\limits_{\{1,i,c\}\in \mathcal{W}_i} \Psi_2(g_1\otimes g_{i}\otimes g_{c})^{\lambda_{\{1,i,c\}}}\Bigr)= ([g_1,g_j] \otimes g_{c(j)})^{\tau\lambda_{\{1,j,c(j)\}}} 
        \]
        where ${\tau=1}$ if ${j<c(j)}$ and ${\tau=-1}$ if ${j>c(j)}$. Similarly, using {Lemma~ \ref{(a,b,c)properties}\textit{(ii)}}, \eqref{Geq:p4}, \eqref{Geq:p1} and \eqref{Geq:p2}, we obtain ${P_{1,j,c(j)}(\Psi_2(g_i\otimes g_{c(i)}\otimes g_{k+2}))=1}$, if ${i\neq j}$. Thus, applying ${P_{1,j,c(j)}}$ to equation \eqref{linearindepenceeqn2 d>k+1} yields 
        \begin{equation} \label{EQ1}
           ([g_1,g_j] \otimes g_{c(j)})^{ \tau\lambda_{\{1,j,c(j)\}}} ([g_1,g_j] \otimes g_{c(j)})^{-\alpha_j\lambda_{\{j,c(j),k+1\}}}=1. 
        \end{equation}
        Applying projection map ${P_{1,c(j),j}: \gamma_2G/\gamma_3G \otimes G/Z \to \langle [g_1,g_{c(j)}] \otimes g_j \rangle}$ to \eqref{linearindepenceeqn2 d>k+1}, we obtain
         \begin{equation} \label{EQ2}
              ([g_1,g_{c(j)}]\otimes g_{j} )^{- \tau\lambda_{\{1,j,c(j)\}}}([g_1,g_{c(j)}] \otimes g_{j})^{\alpha_{c(j)}\lambda_{\{j,c(j),k+2\}}}=1. 
 \end{equation}
 The equations \eqref{EQ1} and \eqref{EQ2} yield ${\lambda_{\{j,c(j),k+2\}}=0}$ for all ${\{j,c(j),k+2\}\in \mathcal{W}}$. By Proposition \ref{Linearindependence}, we have that ${\lambda_{\{1,i,c\}}=0 }$  for all ${\{1,i,c\}\in \mathcal{W}(\mathcal{B})}$. Thus, ${W(\mathcal{B})\cup W}$ is linearly independent. By Proposition \ref{Estimatesize} and Lemma \ref{Lemma r and sum}, we obtain that the size of ${{W}(\mathcal{B})}$ is at least ${\left(\sum\limits_{i=2}^{k+1}d-i\right)}$. Clearly, ${\vert W\vert \geq k-2}$, and ${W(\mathcal{B})\cup W\subseteq\im{\Psi_2}}$. Thus, ${\vert\im{\Psi_2}\vert\geq p^{\left(\sum\limits_{i=2}^{k+1}d-i\right) + (k-2)}}$, and hence the proof follows from \eqref{schur_Psi_image_inequality}.
\end{proof}
\begin{corollary}
            Let $G$ be a special $p$-group of order $p^n$ with $d(G)=d$ and $d(\gamma_2G)=k$, where $k>2$. If $d>k+1$, then
            \[\vert H^2(G,\mathbb{Z}/p\mathbb{Z})\vert \leq  p^{d+\frac{1}{2}(d-1)(n+k)-\big(\sum\limits_{i=2}^{k+1}(d-i)\big) -(k-2)}.\]
        \end{corollary}
        \begin{proof}
            The proof follows mutatis mutandis the proof of Corollary \ref{main bound Zp coefficient }.
        \end{proof}

\section{Construction of $p$-groups of class 2 with $\delta-1\leq k'$ achieving the bound on the size of the Schur multiplier}
In this section, for any natural numbers $d$, $\delta$, $k$, $k'$ satisfying $k=k'$ and $\delta -1 \leq k'$, we construct a capable $p$-group $H$ of nilpotency class two and exponent $p$ with $d(H)=d, $ $d(H/Z)=\delta$, $d(\gamma_2H)=k'$ and $\vert\gamma_2H\vert=p^k$ such that $\vert M(H)\vert$ attains the bound in {Theorem~\ref{mainbound}}. 

We require the next two results to prove the main theorem of the section.

    \begin{theorem} \cite[Theorem 2.5, Theorem 2.7]{MavTho2023} \label{maxgraphcharaterize}
		Let $ \binom{r}{2}  \leq n < \binom{r+1}{2} $ and $t:=n- \binom{r}{2}$. Let $\mathcal{G}$ be the graph with $n$ edges and maximum number of triangles.
		\begin{enumerate}
			\item[(i)] If $t=0 $, then $\mathcal{G}$ is the complete graph $K_r$.
			
			\item[(ii)] If $t=1$ and $\mathcal{G}$ is connected, then $\mathcal{G}$ is the graph obtained by attaching a vertex $v$ to a vertex of the complete graph $K_r$ with $1$ edge.
			
			\item[(iii)] If $t=1$ and $\mathcal{G}$ is disconnected, then $\mathcal{G}$ is the graph $K_2 \cup K_r$.
			
			\item[(iv)] If $t \neq 0,1$, then $\mathcal{G}$ is the graph obtained by attaching a vertex $v$ with $t$ edges to the complete graph $K_r$, where the $t$ edges incident to $t$ vertices of $K_r$.
			
		\end{enumerate}
        Note that each of the graphs above has $\binom{r}{3}+\binom{t}{2}$ triangles.
    \end{theorem}

    \begin{prop} \label{Elliscapability} \cite[Proposition 9]{Elli1998capability}
        Let $G$ be a finitely generated group of nilpotency class two and of prime exponent. Let $\{x_1,\ldots, x_k\}$ be a subset of $G$ corresponding to a basis of the vector space $G/Z$, and suppose that those non-trivial commutators of the form $[x_i, x_j]$ with $1 \leq i < j \leq k$ are distinct and constitute a basis for the vector space $[G, G]$. Then $G$ is capable.
    \end{prop}


    Now we come to the main theorem of this section.
\begin{theorem}\label{sharpness thm}

       Given any natural numbers $d$, $\delta$, $k$, $k'$ satisfying $k=k'$ and $\delta -1 \leq k'$, there exists a capable $p$-group $H$ of nilpotency class two and exponent $p$ with $d(H)=d, $ $d(H/Z)=\delta$, $d(\gamma_2H)=k'$ and $\vert\gamma_2H\vert=p^k$ such that $\vert M(H)\vert$ attains the bound in {Theorem~\ref{mainbound}}.
    \end{theorem}   
    \begin{proof}
        The proof is divided into two steps. In the first step, we will construct a $p$-group $G$ and prove the theorem under the assumption $d=\delta$. In the second step, we will show that $H=G \times (\mathbb{Z}/p\mathbb{Z})^{d-\delta}$ is a capable group and that $\vert M(H)\vert$ attains the bound in Theorem \ref{mainbound}.
        
        \textbf{Step 1: Case $d=\delta$}

Let $K$ be a special $p$-group with exponent $p$ and rank $\binom{\delta}{2}$ minimally generated by elements $x_1,\ldots,x_\delta$. Let $r$ and $t$ be non-negative integers such that ${\binom{\delta}{2}-k=\binom{r}{2}+t}$, where $0\leq t <r$. Let $N$ be the subgroup of $K$ defined by
        \begin{align*}
            N=\Bigl\langle [x_{\delta-r},x_{\delta-t+1}],\ldots,[x_{\delta-r},x_{\delta}], [x_{i},x_{j}] \mid \delta-r+1 \leq i <j \leq \delta \Bigr\rangle.
        \end{align*}
        Note that $N$ is a normal subgroup of $K$ because $N \leqslant Z(K)$. The subgroup $\gamma_2K$ is a vector space over $\mathbb{F}_p$, and the generators of $N$ are linearly independent. Therefore, $\vert N\vert = p^{\binom{r}{2}+t}=p^{\binom{\delta}{2}-k}$. Set $G=K/N$, and observe that $|G|=p^{\delta+k}$, $d(G)=\delta$, $|\gamma_2G|=p^k$ and $d(\gamma_2G/\gamma_3G)=k$. It is easy to see that $\{x_1N,\ldots,x_{\delta}N\}$ is a minimal generating set of $G$, and we denote $x_iN$ by $g_i$ for all $i$ with $1\leq i \leq \delta$.  The groups $G/\gamma_2G$ and $\gamma_2G$ are vector spaces over $\mathbb{F}_p$, and note that $\{g_1\gamma_2G,\ldots,g_{\delta}\gamma_2G\}$ forms a basis of $G/\gamma_2G.$ We now claim that $d(G/Z)=\delta$. Since $G$ is a group of nilpotency class two and exponent $p$, we have $\Phi(G)=\gamma_2G\leqslant Z(G)$. Since $d(G)= \dim_{\mathbb{F}_p} (G/\Phi(G))$ and $d(G/Z)= \dim_{\mathbb{F}_p}(G/Z)$, it is enough to show that $\gamma_2G=Z(G)$. Observe that any element $x \in G$ can be expressed as $x = \left(\prod\limits_{i=1}^{\delta} g_i^{\beta_i}\right)y$, where $\beta_i$ are integers satisfying $0 \leq \beta_i < p$ and $y \in \gamma_2(G)$.
        If $x\in G\setminus\gamma_2G$, then $\beta_i \neq 0$ for some $i$ with $1 \leq i \leq \delta$. If $\beta_i=0$ for all $i$ with $2\leq i\leq \delta$, then $\beta_1\neq 0$. Thus, $[x,g_2]=[g_1^{\beta_1},g_2]=[g_1,g_2]^{\beta_1}$ is a non-trivial element of $G$, and  hence $x\notin Z(G)$. If $\beta_i\neq 0$ for some $i$ with $2\leq i\leq \delta$,  then $[x,g_1]=\prod\limits_{i=2}^{\delta}[g_i,g_1]^{\beta_i}$ is non-trivial, since the elements $[g_i,g_1]$ are linearly independent elements of $\gamma_2G$. Thus, $x\notin Z(G)$, thereby proving that $\gamma_2(G)=Z(G)$ and $d(G/Z)=\delta$. Note that if $\delta-1> k'$, the construction in this proof does not produce a group $G$ with $d(G)=d(G/Z)$, as the elements $[g_i,g_1]$ need not be linearly independent in that case. For notational convenience, we denote $g_i\gamma_2G$ as $g_i$, which will be clear from the context, and no confusion shall arise. Note that the commutator map $A:G/\gamma_2G\times G/\gamma_2G\to \gamma_2G$ is an alternating map. Let $g_1<\ldots<g_{\delta}$ be an order on the basis of $G/\gamma_2G$, $\mathcal{B}$ and $B$ constructed as in Section \ref{Construction_Basis}. Figure \ref{fig:B elements} illustrates the elements of $\mathcal{B}$ and those not contained in $\mathcal{B}$. The elements highlighted in light pink represent those in $\mathcal{B}$, while the elements highlighted in blue and green, correspond to those not in $\mathcal{B}$.
\begin{figure}
    \centering
    \sbox0{\begin{tabular}{rl}
  \tikz{\node [fill=blue!30!white, opacity=0.5, inner sep=1ex]{};} & $\binom{r}{2}$ elements\\[4.5pt]
  \tikz{\node [fill=green!30!white, opacity=0.5, inner sep=1ex]{};} & $t$ elements \\[4.5pt]
    \tikz{\node [fill=red!30!white, opacity=0.5, inner sep=1ex]{};} & $\binom{\delta}{2}-\binom{r}{2}-t$ elements
\end{tabular}}
\begin{tikzpicture}
 \matrix (m) [matrix of math nodes, nodes in empty cells, row sep=1em, column sep=1em]
  {
     \scalebox{.7}{$\{1,2\}$} \\
     \scalebox{.7}{$\{1,3\}$} &\cdots \\
     \cdots &\cdots &\scalebox{.7}{$\{{\delta-r},{\delta-r+1}\}$}\\
     \cdots &\cdots & \cdots & \scalebox{.7}{$\{{\delta-r+1},{\delta-r+2}\}$}\\
     \cdots &\cdots &\scalebox{.7}{$\{{\delta-r},{\delta-t}\}$} & \cdots & \cdots\\
     \cdots &\cdots &\scalebox{.7}{$\{{\delta-r},{\delta-t+1}\}$} & \cdots & \cdots & \cdots\\
     \cdots & \cdots & \cdots & \cdots & \cdots & \cdots & \cdots \\
     \scalebox{.7}{$\{1,\delta\}$} & \cdots &\scalebox{.7}{$\{{\delta-r},{\delta}\}$}&\scalebox{.7}{$\{\delta-r+1,\delta\}$} &\cdots &\cdots&  \cdots & \scalebox{.7}{$\{{\delta-1},{\delta}\}$}.\\
  };
  
  \draw[red!30!white, thick, rounded corners=8pt, fill=red!30!white, opacity=0.5] (-7.2, -3.8) -- (-7.2, 3.7) -- (-6.2,3.7) -- (-4.6, 2.8) -- (-4.6, -3.8) -- cycle;

  \draw[red!30!white, thick, rounded corners=8pt, fill=red!30!white, opacity=0.5] (-4.4, -.8) -- (-4.4, 1.7) -- (-1.9,1.7) -- (-1.9, -.8) -- cycle;
  \draw[green!30!white, thick, rounded corners=8pt, fill=green!30!white, opacity=0.5] 
    (-4.4, -3.8) -- (-4.4, -1) -- (-1.9,-1) -- (-1.9, -3.8) -- cycle;

  \draw[blue!30!white, thick, rounded corners=8pt, fill=blue!30!white, opacity=0.5] 
    (-1.6, -3.8) -- (-1.6, 1) -- (1.4, 1) -- (8, -3.8) -- cycle;

\node[xshift=3.0cm, yshift=2.5cm, draw, rounded corners] {\usebox0};
\end{tikzpicture}
    \caption[labelsep=space]{}
    \label{fig:B elements}
\end{figure}
        Recall that $\{i,j\} \in \mathcal{B}$ implies that $[g_i,g_j]$ is an element of the basis $B$ of the vector space $\gamma_2G$. Thus, 
        \[{B=\bigl\{ [g_{i},g_{j}], [g_{\delta-r},g_{\delta-r+1}],\ldots, [g_{\delta-r},g_{\delta-t}] \mid 1 \leq i  \leq \delta-r-1, i < j \leq \delta\bigr\}}\]
        forms a basis for $\gamma_2G$. Moreover, the elements of 
        \[\bigl\{[g_{\delta-r},g_{\delta-t+1}],\ldots,[g_{\delta-r},g_{\delta}], [g_{i},g_{j}] \mid \delta-r+1 \leq i <j \leq \delta \bigr\}\]
        are trivial. Observe that $\im{\Psi_2}$ defined in \eqref{eqn:Psi_2_definition} is generated by the set ${\{\Psi_2(g_i \otimes g_j \otimes g_k) \mid 1 \leq i,j,k \leq \delta \}}$. By Lemma \ref{Psiproperties}$(i)$, $\Psi_2(g_i \otimes g_j \otimes g_k)$ is trivial when $\{i,j,k\}$ does not contain mutually distinct elements. Applying {Lemma~\ref{Psiproperties}$(ii)$}, we observe that  $\im{\Psi_2}$ is generated by $\binom{\delta}{3}$ elements ${\{\Psi_2(g_i \otimes g_j \otimes g_k) \mid 1 \leq i < j <k \leq \delta \}}$. Let $\mathcal{W}(\mathcal{B})$ and $W(\mathcal{B})$ be defined as in \eqref{eq:scriptW(scriptB)} and \eqref{eq:W(B)}, respectively. Set $\mathcal{T}=\{\{i,j,k\} \mid 1 \leq i < j <k \leq \delta\} $. Note that $\Psi(g_i \otimes g_j \otimes g_k)=1$ for all $\{i,j,k\} \in \mathcal{T} \setminus \mathcal{W}(\mathcal{B})$, and hence $\im{\Psi_2}$ is generated by $W(\mathcal{B})$. By Proposition~\ref{Linearindependence}, $W(\mathcal{B})$ is linearly independent, and thus it is a basis of $\im{\Psi_2}$. By Proposition \ref{Estimatesize}, $\dim(\im{\Psi_2}) \geq \binom{\delta}{3}-\binom{r}{3}-\binom{t}{2}$. We will show that ${\dim(\im{\Psi_2}) = \binom{\delta}{3}-\binom{r}{3}-\binom{t}{2}}$. By Proposition \ref{linearindepenceeqn}, there is a bijection between $W(\mathcal{B})$ and $\mathcal{W}(\mathcal{B}),$ and thus ${\dim\im{\Psi_2}=\vert \mathcal{W}(\mathcal{B})\vert}$. We will now determine $|\mathcal{T} \setminus \mathcal{W}(\mathcal{B})|$. Towards this aim, we consider a graph with $r+1$ vertices $\{\delta-r,\delta-r+1, \ldots , \delta\}$, and an edge between vertices $i$ and $j$ if $\{i,j\} \notin \mathcal{B}$. There is a bijection between triangles in this graph and the set $\mathcal{T} \setminus \mathcal{W}(\mathcal{B})$. The subgraph formed by the set of vertices $\{\delta-r+1,\ldots,\delta\}$ is a complete graph on $r$ vertices. Note that every pair of vertices $\{i,j\}$ corresponding to an edge in this subgraph is highlighted in blue in Figure \ref{fig:B elements}. The remaining vertex $\delta-r$ is connected to the $t$ vertices $\{\delta-t+1, \ldots, \delta\}$. Note that the pairs of vertices $\{\delta-r,\delta-t+1\},\ldots,\{\delta-r,\delta\}$ corresponding to the edges adjacent to $\delta-r$ are highlighted in green in Figure~\ref{fig:B elements}. By Theorem~\ref{maxgraphcharaterize}$(i)$, $(ii)$ and $(iv)$, this graph has $\binom{r}{3}+\binom{t}{2}$ triangles. Thus, ${\dim\im{\Psi_2}= \vert\mathcal{W}(\mathcal{B})\vert=\binom{\delta}{3}-\binom{r}{3}-\binom{t}{2}}$, and hence $|\im{\Psi_2}|=p^{\binom{\delta}{3}-\binom{r}{3}-\binom{t}{2}}$.
        By \cite[page 101]{BlaEve1979}, we have the exact sequence,  
        \begin{equation} \label{eq:Schurexactsequence}
		1 \to X \to G/\gamma_2G \otimes \gamma_2G \to M(G) \to M(G/\gamma_2G) \to \gamma_2G \to 1,
        \end{equation}
        where $X=\im{\Psi_2}$ for $p$-groups of nilpotency class $2$ and exponent $p$.
        Using \eqref{eq:Schurexactsequence}, we obtain
        \begin{equation*}
            |M(G)|=\frac{|G/\gamma_2G \otimes \gamma_2G||M(G/\gamma_2G)|}{|\im{\Psi_2}||\gamma_2G|}.
        \end{equation*}
        Since $G/\gamma_2G$ and $\gamma_2G$ are elementary abelian groups of rank $\delta$ and $k$ respectively, we have that ${|G/\gamma_2G \otimes \gamma_2G|=p^{\delta k}}$. Using \cite[Corollary 2]{Berk1991}, we obtain that $|M(G/\gamma_2G)|=p^{\frac{\delta(\delta-1)}{2}}$. Noting that ${|\im{\Psi_2}|=p^{\binom{\delta}{3}-\binom{r}{3}-\binom{t}{2}}}$ and $|\gamma_2G|=p^k$, we observe that 
        \begin{equation} \label{Schur multiplier achieved for d=delta}
            |M(G)|=p^{\binom{\delta}{2}+\delta k-k-\binom{\delta}{3}+\binom{r}{3}+\binom{t}{2}}.
        \end{equation}
        Thus, $|M(G)|$ attains the bound in Theorem \ref{mainbound} and hence Step 1 is complete.
        
        \textbf{Step 2: General case $d \geq \delta$}

        Consider the group $H=G \times (\mathbb{Z}/p\mathbb{Z})^a$, where $a=d-\delta$. It is easy to verify that $|H|=p^{\delta+k+a}$, $d(H)=d(G)+a=\delta+a$, $d(H/Z)=\delta$, ${d(\gamma_2H/\gamma_3H)=k'=k}$ and $|\gamma_2H|=p^k$. 
        Using \cite[Theorem 2.2.10]{Karp1987}, we obtain that 
        \[|M(H)|=|M(G)| |M((\mathbb{Z}/p\mathbb{Z})^a)||G/\gamma_2G \otimes_{\mathbb{Z}} (\mathbb{Z}/p\mathbb{Z})^a|.\] By  \eqref{Schur multiplier achieved for d=delta}, we have $|M(G)|=p^{\binom{\delta}{2}+\delta k-k-\binom{\delta}{3}+\binom{r}{3}+\binom{t}{2}}$. Using \cite[Corollary 2]{Berk1991}, we have $|M((\mathbb{Z}/p\mathbb{Z})^a)|=p^{\binom{a}{2}}$. Moreover, $|G/\gamma_2G \otimes_{\mathbb{Z}} (\mathbb{Z}/p\mathbb{Z})^a|=p^{da}$. Therefore, 
        \begin{align*}
            |M(H)|&=p^{\binom{\delta}{2}+\delta k-k-\binom{\delta}{3}+\binom{r}{3}+\binom{t}{2}+\binom{a}{2}+da}\\
            &=p^{\frac{1}{2}(d-1)(d+2k)-k(d-\delta)-\binom{\delta}{3}+ \binom{r}{3} + \binom{t}{2}}.
        \end{align*}
        Thus, $M(H)$ attains the bound in Theorem \ref{mainbound}. By Proposition \ref{Elliscapability}, it follows that $H$ is capable, and hence the proof.
    \end{proof}
\section*{Acknowledgements}  Sathasivam K acknowledges the Ministry of Education,  Government of India, for the doctoral fellowship under the Prime Minister's Research Fellows (PMRF) scheme PMRF-ID 0801996.

\bibliographystyle{amsplain}
\bibliography{Bibliography}

\end{document}